\documentclass[12pt]{amsart}

\usepackage{amsmath}
\usepackage{amsthm}
\usepackage{amssymb}
\usepackage{mathrsfs}
\usepackage{ifthen}
\usepackage{graphicx}
\usepackage{enumerate}
\usepackage{hyperref}
\usepackage{float}
\usepackage{tikz}
\usepackage{enumitem}
\usepackage{fullpage}
\usepackage{enumerate}
\usepackage{enumitem}
\usepackage{xcolor}
\usepackage{todonotes}

 \nonstopmode \numberwithin{equation}{section}

\theoremstyle{plain}

\newtheorem{Thm}{Theorem}[section]

\newtheorem{lemma}[Thm]{Lemma}
\newtheorem{prop}[Thm]{Proposition}

\newtheorem{rem}[Thm]{Remark}
\newtheorem{appendixlemma}{Lemma}[section]

\theoremstyle{definition}

\newtheorem{defn}[Thm]{Definition}

\newtheorem{matrixcondition}[Thm]{Condition}

\newcommand{\conjugate}{{\operatorname{^*}}}
\newcommand{\Tr}{\operatorname{Tr}}
\newcommand{\Var}{\operatorname{Var}}
\newcommand{\Cov}{\operatorname{Cov}}
\newcommand{\norm}[1]{\left\lVert#1\right\rVert}

\allowdisplaybreaks

\title{Spectrum of random centrosymmetric matrices; CLT and Circular law}

\author{Indrajit Jana}
\address{Indian Institute of Technology, Bhubaneswar}
\email{ijana@iitbbs.ac.in}
\thanks{Indrajit Jana - \textit{Email: ijana@iitbbs.ac.in;}
}
\author{Sunita Rani}
\address{Indian Institute of Technology, Bhubaneswar}
\email{s21ma09007@iitbbs.ac.in}
\thanks{Sunita Rani - \textit{Email: s21ma09007@iitbbs.ac.in;} (Corresponding author).
}

\begin{document}
\date{}

\begin{abstract}
We analyze the asymptotic fluctuations of linear eigenvalue statistics of random centrosymmetric matrices with i.i.d. entries. We prove that for a complex analytic test function, the centered and normalized linear eigenvalue statistics of random centrosymmetric matrices converge to a normal distribution. We find the exact expression of the variance of the limiting normal distribution via combinatorial arguments. Moreover, we also argue that the limiting spectral distribution of properly scaled centrosymmetric matrices follows the circular law.

keywords: Centrosymmetric matrix, Linear eigenvalue statistics, Central limit theorem, Circular law.
\end{abstract}

\maketitle

%%%%%%%%%%%%%%%%%%%%%%%%%%%%%%%%%%%%%%%%%%%%%%%%%%%%%%%%%%%%%%%%%%%%%%%%%%%
%%%%%%%%%%%%%%%%%%%%%%%%%%%%  Introduction  %%%%%%%%%%%%%%%%%%%%%%%%%%%%%%%
%%%%%%%%%%%%%%%%%%%%%%%%%%%%%%%%%%%%%%%%%%%%%%%%%%%%%%%%%%%%%%%%%%%%%%%%%%%

\section{Introduction} Let $M$ be an $n \times n$ matrix with real or complex entries. The empirical spectral distribution (ESD) is defined by the measure
\begin{align*}
    \mu_{n}(\cdot) = \frac{1}{n}\sum_{i=1}^{n}\delta_{\lambda_{i}}(\cdot),
\end{align*}
where $\lambda_{1}, \lambda_{2}, \ldots,\lambda_{n}$ are the eigenvalues of the matrix $M$, and $\delta_{x}(\cdot)$ is the point measure at $x\in \mathbb{C}$. The linear eigenvalue statistics (LES) of $M$ corresponding to a test function $f$ is defined by
$$
\operatorname{L}(f)= \sum_{k=1}^n f\left(\lambda_k\right).
$$
The function $f$ is referred to as the test function. The asymptotic behaviour of $\mu_{n}(\cdot)$ and $a_{n}(\operatorname{L}(f)-b_{n})$ for some sequences $ \langle a_{n}\rangle$ and  $\langle b_{n} \rangle $ have been studied in depth for different types of random matrices. We give a brief overview below.

Jonsson \cite{Jonsson1982SomeLT} established the central limit theorem (CLT) for LES in Wishart matrices. Subsequently, researchers extensively studied the fluctuation of eigenvalues of various random matrices. Some of the key contributions include Johansson \cite{johansson1998fluctuations}, Sinai and Soshnikov \cite{sinai1998central}, Bai and Silverstein \cite{bai2008clt}, Lytova and Pastur \cite{lytova2009central},  Shcherbina \cite{shcherbina2011central}. Furthermore, the same has been studied for several structured matrices, such as band and sparse matrices \cite{anderson2006clt, li2013central, shcherbina2015fluctuations, jana2016fluctuations}, Toepliz and band Toeplitz matrices \cite{chatterjee2009fluctuations, liu2012fluctuations}, circulant matrices \cite{adhikari2017fluctuations, adhikari2018universality, maurya2021fluctuations}.

The CLT of LES for non-Hermitian matrices have been studied in different set ups as well; such as when the test function is analytic \cite{rider2006gaussian, o2016central, jana2022clt}, the test function is non-analytic but the moments of the matrix entries match with that of a Ginibre ensemble up to fourth moment \cite{kopel2015linear}. Later, it was generalized for continuous bounded test functions \cite{cipolloni2021fluctuation, cipolloni2022fluctuations, cipolloni2023central}. A more comprehensive list of recent results can be found on the review article by Forrester \cite{forrester2023review}.

In this article, we study the spectral properties of random centrosymmetric matrices. A centrosymmetric matrix is symmetric around its center which is formally defined in the Definition \ref{centrodefn}. Centrosymmetric matrices appeared in many different contexts in mathematics, such as solutions to chessboard separation problems \cite{chatham2012centrosymmetric}, stochastic matrices \cite{cao2022centrosymmetric}, and in various other contexts \cite{cruse1977some, tao2002spectral, rojo2004some, bai2005inverse}. However, to the best of our knowledge the spectral properties of random centrosymmetric matrices have not been studied yet. In this paper, we investigate the limiting ESD of a random centrosymmetric matrix and the fluctuations of the LES.

%%%%%%%%%%%%%%%%%%%%%%%%%%%%%%%%%%%%%%%%%%%%%%%%%%%%%%%%%%%%%%%%%%%%%%%%%%%
%%%%%%%%%%%%%%%%%%%%%%%%%%%%%   Notations  %%%%%%%%%%%%%%%%%%%%%%%%%%%%%%%% 
%%%%%%%%%%%%%%%%%%%%%%%%%%%%%%%%%%%%%%%%%%%%%%%%%%%%%%%%%%%%%%%%%%%%%%%%%%%

\section{Notations and layout}

This article is organized as follows. We state the main theorem in Section \ref{sec:main result}. The proof of the CLT is shown in Section \ref{sec: proof of the clt}, and the variance calculation of the limiting normal distribution is delegated to Section \ref{sec:variance calculation}. In addition, it was also argued that the limiting spectral distribution follows the circular law, which was demonstrated in Sectoin \ref{sec:circular law proof}.

Throughout this article, the identity matrix is denoted by $I$, and the counter-identity matrix is denoted by $J$, as defined in Definition \ref{dfn: counter-identity}. The dimensions of such matrices should be understood from the context. The eigenvalues of an $n\times n$ matrix $M$ are denoted by $\lambda_{1}(M),$  $\lambda_{2}(M),$ $\ldots,$ $\lambda_{n}(M).$ The complex conjugate transpose of a matrix $Y$ is denoted by $Y\conjugate$, whereas $\bar{Y}$ stands for only the complex conjugate of $Y$ without taking transpose. The notation $\mathcal{R}_{z}(M)=(zI-M)^{-1}$ represents the resolvent of the square matrix $M$. The notations $\zeta_{n}\stackrel{d}{\to}\zeta$, $\zeta_{n}\stackrel{p}{\to}\zeta$, and $\zeta_{n}\stackrel{a.s.}{\to}\zeta$ are used to indicate that the sequence of random variables $\{\zeta_{n}\}$ converges to $\zeta$ in distribution, in probability, and almost surely, respectively.

For any random variable $\xi$, $\xi^{\circ}$ denotes the centered random variable $\xi-\mathbb{E}[\xi].$ Moreover, if $\xi$ depends on a random matrix, then $\xi_{k}^{\circ}:=\xi-\Xi_{k}[\xi]$, where $\Xi_{k}[\cdot]$ signifies averaging with respect to the $k$-th column of the underlying random matrix.

The vectors $e_{1}, e_{2}, \ldots, e_{n}$ denote the canonical basis vectors of $\mathbb{R}^{n}.$ The disc of radius $r$, denoted by $\mathbb{D}_{r}$, is the set $\{z\in \mathbb{C}:|z|\leq r\}$, and its boundary is denoted by $\partial \mathbb{D}_{r}$. The notations $K$ or $C$ are used to denote a universal constant whose exact value may change from line to line.

%%%%%%%%%%%%%%%%%%%%%%%%%%%%%%%%%%%%%%%%%%%%%%%%%%%%%%%%%%%%%%%%%%%%%%%%%%%
%%%%%%%%%%%%%%%%%%%%%%%%%%%%%%   Theorem  %%%%%%%%%%%%%%%%%%%%%%%%%%%%%%%%% 
%%%%%%%%%%%%%%%%%%%%%%%%%%%%%%%%%%%%%%%%%%%%%%%%%%%%%%%%%%%%%%%%%%%%%%%%%%%

\section{Main result}\label{sec:main result}
Let us first define the centrosymmetric matrix.
\begin{defn}[Centrosymmetric  Matrix]\label{centrodefn} A random Centrosymmetric matrix is a random matrix which is symmetric about its center. Formally, let $M = [m_{i,j}]_{n \times n}$ be a random matrix, where $m_{ij}$ are i.i.d. random variables subject to the following constraint
$$
m_{i, j}=m_{n-i+1, n-j+1} \text { for } i, j \in\{1, \ldots, n\} .
$$
Then $M$ is called a random centrosymmetric matrix.
    \end{defn}

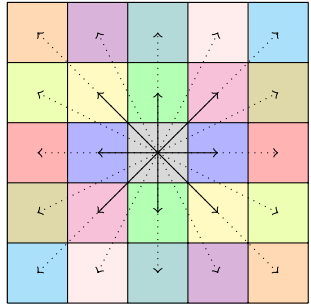
\begin{figure}[h]
\centering
\begin{tikzpicture}[scale=0.8]
    \draw (0,0) rectangle (5,5);
    \foreach \y in {1,2,3,4}{
        \draw (0,\y) -- (5,\y);
    }
    \foreach \x in {1,2,3,4}{
        \draw (\x,0) -- (\x,5);
    }
    % Fill color in the first row
    \fill[cyan!30] (5,5) rectangle (4,4);
    \fill[pink!30] (4,5) rectangle (3,4);
    \fill[teal!30] (3,5) rectangle (2,4);
    \fill[violet!30] (2,5) rectangle (1,4);
    \fill[orange!30] (1,5) rectangle (0,4);
    
    % Fill color in the last row (reversed order)
    \fill[cyan!30] (0,0) rectangle (1,1);
    \fill[pink!30] (1,0) rectangle (2,1);
    \fill[teal!30] (2,0) rectangle (3,1);
    \fill[violet!30] (3,0) rectangle (4,1);
    \fill[orange!30] (4,0) rectangle (5,1);
     % Fill color in the second row
     \fill[olive!30] (5,4) rectangle (4,3);
    \fill[magenta!30] (4,4) rectangle (3,3);
    \fill[green!30] (3,4) rectangle (2,3);
    \fill[yellow!30] (2,4) rectangle (1,3);
    \fill[lime!30] (1,4) rectangle (0,3);
     % Fill color in the forth row (reversed order)
     \fill[olive!30] (0,1) rectangle (1,2);
    \fill[magenta!30] (1,1) rectangle (2,2);
    \fill[green!30] (2,1) rectangle (3,2);
    \fill[yellow!30] (3,1) rectangle (4,2);
    \fill[lime!30] (4,1) rectangle (5,2);
    % Fill color in the third row
    \fill[red!30] (0,2) rectangle (1,3);
    \fill[blue!30] (1,2) rectangle (2,3);
    \fill[gray!30] (2,2) rectangle (3,3);
    \fill[blue!30] (3,2) rectangle (4,3);
    \fill[red!30] (4,2) rectangle (5,3);
    
    % Draw grid lines
    \draw (0,0) grid (5,5);
    
    % Draw arrows
    \draw[<->] (1.5,1.5) -- (3.5,3.5);
    \draw[dotted, <->] (0.5,0.5) -- (4.5,4.5);
    \draw[<->] (3.5,1.5) -- (1.5,3.5);
    \draw[dotted, <->] (4.5,0.5) -- (0.5,4.5);
    \draw[<->] (1.5,2.5) -- (3.5,2.5);
    \draw[dotted, <->] (0.5,2.5) -- (4.5,2.5);
    \draw[<->] (2.5,1.5) -- (2.5,3.5);
    \draw[dotted, <->] (2.5,0.5) -- (2.5,4.5);
    \draw[dotted, <->] (1.5,0.5) -- (3.5,4.5);
    \draw[dotted, <->] (3.5,0.5) -- (1.5,4.5);
    \draw[dotted, <->] (0.5,1.5) -- (4.5,3.5);
    \draw[dotted, <->] (4.5,1.5) -- (0.5,3.5);
\end{tikzpicture}
\caption{Symmetry pattern of a centrosymmetric 5$\times$5 matrix.}
\label{fig:cento_fig}
\end{figure}

  \begin{defn}[Poincaré inequality]{\label{def:Poincaré inequality}} A complex random variable $\zeta$ is said to satisfy the Poincar\'e inequality with a positive constant $\alpha$ if for any differentiable function $g:\mathbb{C}\to\mathbb{C}$, we have $\Var(g(\zeta))\leq \frac{1}{\alpha}\mathbb{E}[|g'(\zeta)|^{2}].$ Here, $\mathbb{C}$ is identified with $\mathbb{R}^{2}$ and $|g'(\zeta)|$ represents the $\ell^{2}$ norm of the two dimensional vector $g'$. Here are some characteristics of the Poincar\'e inequality.
  \begin{enumerate}
      \item If $\zeta$ satisfies the Poincar\'e inequality with constant $\alpha$, then $\kappa \zeta$ also satisfies the Poincar\'e inequality with constant $\alpha/\kappa^{2}$ for any nonzero real constant $\kappa.$
      \item If two independent random variables $\zeta_{1}, \zeta_{2}$ satisfy the Poincar\'e inequality with the same constant $\alpha$, then $\zeta=(\zeta_{1}, \zeta_{2})$ also satisfies the Poincar\'e inequality with the same constant $\alpha$. Here $\zeta$ is treated as a random vector and, accordingly, $g'$ in the definition of Poincar\'e inequality will be replaced by $\nabla g$.
      \item \cite[Lemma 4.4.3]{anderson2010introduction} If $\zeta\in \mathbb{C}^{n}$ satisfies the Poincar\'e inequality with a constant $\alpha$, then for any differentiable function $g:\mathbb{C}^{n}\to\mathbb{C}$,
      \begin{align*}
          \mathbb{P}(|g(\zeta) - \mathbb{E}[g(\zeta)]| > t)\leq C\exp\left\{-\frac{\sqrt{\alpha}t}{\sqrt{2}\|\|\nabla g\|_{2}\|_{\infty}}\right\},
      \end{align*}
      where $C=-2\sum_{k=1}^{\infty}2^{k}\log(1-2^{-2k-1}).$ Here $\mathbb{C}^{n}$ is identified with $\mathbb{R}^{2n}$. Moreover, $\|\nabla g(x)\|_{2}$ is the $\ell^{2}$ norm of the $2n$ dimensional vector $\nabla g(x)$ at $x\in \mathbb{C}^{n}$, and $\|\|\nabla g\|_{2}\|_{\infty}=\sup_{x\in \mathbb{C}^{n}}\|\nabla g(x)\|_{2}.$
    \end{enumerate}
\end{defn}  

For instance, a Gaussian random variable satisfies the Poincar\'e inequality.

\begin{matrixcondition}\label{cond:matrixcond}
  Let $M=[m_{ij}]_{n\times n}$ be a centrosymmetric matrix, where 
  \begin{align*}
      m_{ij} = \frac{1}{\sqrt{n}}x_{ij}.
  \end{align*}
  Assume that $\{x_{ij}: 1\leq i,j\leq n\}$ are i.i.d. complex random variables subject to the constraint $x_{ij} = x_{n+1-i, n+1-j}$. In addition, $x_{ij}$s satisfy the following conditions.
  \begin{enumerate}[label=(\roman*)]
  \item $\mathbb{E}[x_{ij}] = 0, \mathbb{E}[x_{ij}^2] = 0$ and $\mathbb{E}[|x_{ij}|^2] = 1$ for all $1\leq i, j \leq n,$
  %\item \hl{All the odd moments are zero.}
  \item $x_{ij}$s are continuous random variables with bounded density and satisfy the Poincaré inequality with some universal constant $\alpha$.
\end{enumerate}
\end{matrixcondition}

Before stating our main theorem, let us define the centered LES as 
\begin{align*}
    \operatorname{L}_{f}^{\circ}(M)=\sum_{k=1}^{n}f(\lambda_{k}(M))-\sum_{k=1}^{n}\mathbb{E}[f(\lambda_{k}(M))].
\end{align*}
We are now ready to state our main theorem.
\begin{Thm}\label{thm: main theorem}
Suppose $M$ is a centrosymmetric random matrix satisfying Condition \ref{cond:matrixcond}. Let $f:\mathbb{C}\to\mathbb{C}$ be an analytic function. Then $\operatorname{L}_{f}^{\circ}(M)\stackrel{d}{\to}N(0, \sigma_{f}^{2})$, where
\begin{align*}
    \sigma_{f}^{2}=\frac{2}{\pi}\int_{\mathbb{D}_{1}}f'(w)\overline{f'(w)}\;d\mathfrak{R}(w)d\mathfrak{I}(w).
\end{align*}
In particular, if $f$ is a polynomial given by $P_{d}(z)=\sum_{k=0}^{d}a_{k}z^{k}$, then $\sigma_{p_d}^2=\sum_{k=1}^{d}2ka_{k}^2$.
\end{Thm}

%%%%%%%%%%%%%%%%%%%%%%%%%%%%%%%%%%%%%%%%%%%%%%%%%%%%%%%%%%%%%%%%%%%%%%%%%%%
%%%%%%%%%%%%%%%%%%%%%%%%%%%%%   Truncation  %%%%%%%%%%%%%%%%%%%%%%%%%%%%%%% 
%%%%%%%%%%%%%%%%%%%%%%%%%%%%%%%%%%%%%%%%%%%%%%%%%%%%%%%%%%%%%%%%%%%%%%%%%%%

\section{Truncation and reduction}\label{sec:truncation}

Before proceeding to the technical proof of the main theorem, we need to recall one important theorem regarding centrosymmetric matrices.

\begin{defn}[Counter-identity Matrix]\label{dfn: counter-identity} It is defined as the square matrix whose counter-diagonal elements are $1$ and all the other elements are zero.
\end{defn}

Throughout this article, we shall denote the counter-identity matrix by $J$. The size of $J$ should be understood from the context. Notice that 
\begin{align*}
J=J^{T}\\
J^{2}= I,
\end{align*}
where $I$ the identity matrix of the same dimension. Now, we invoke the following theorem which reduces the centrosymmetric matrix to block diagonal matrix.

\begin{Thm}\cite[Theorem 9]{weaver1985centrosymmetric}\label{orthogonal}
 \begin{enumerate}[label=(\alph*)]
 % Even sized case
     \item  If $M=\left(\begin{array}{ll}A & B \\ C & D\end{array}\right)$ is an $n \times n$ centrosymmetric matrix with $n=2s$ and $A, B, C$, and $D \text{ are } s\times$ s matrices, then $M$ is orthogonally similar to
     \begin{align*}
         Q^T M Q=\left(\begin{array}{cc}
    A+J C & 0 \\
    0 & A-J C
\end{array}\right),
     \end{align*}
     where 
     \begin{align*}
         Q  =  \sqrt{\frac{1}{2}}\left(\begin{array}{ll}I & -J \\ J & I\end{array}\right).
     \end{align*}

% Odd sized case
     \item If $M=\left(\begin{array}{lll}A & x & B \\ y & q & yJ \\ C & Jx & D\end{array}\right)$ is an $n \times n$ centrosymmetric matrix with $n =2s+1$, and $A, B, C, D$  are  $s \times s$ matrices, $x$ is $s \times 1$ matrix, $y$ is $1 \times s$ matrix, and $q$ is a scalar, then $M$ is orthogonally similar to
     \begin{align*}
         Q^{T} M Q=\left(\begin{array}{ccc}
			A+J C & \sqrt{2} x & 0 \\
			\sqrt{2} y & q & 0 \\
			0 & 0 & A-J C
		\end{array}\right),
     \end{align*}
		where
  \begin{align*}
      Q  =  \sqrt{\frac{1}{2}}\left(\begin{array}{lll} I & 0 & -J \\ 0 & \sqrt{2} & 0\\ J & 0 & I\end{array}\right).
  \end{align*}
 \end{enumerate}
  In the definitions of $Q$ above, $I$ is an $s\times s$ identity matrix and $J$ is an $s\times s$ counter-identity matrix. Moreover, notice that in both the cases $Q^{T}Q=I=QQ^{T}.$
\end{Thm}

Since the $Q^{T}MQ$ is orthogonally similar to the matrix $M$, the eigenvalues of $Q^{T}MQ$ are the same as those of $M$. The remaining part of this article analyses the eigenvalues of $Q^{T}MQ$. For convenience, we adopt the following simplified notation for $Q^{T}MQ$.
\begin{align*}
    Q^{T}MQ = \left[\begin{array}{cc}
        T_{1} & 0 \\
        0 & T_{2}
    \end{array}\right],
\end{align*}
where $T_{2}=A-JC$ and 
\begin{align*}
    T_{1}=\left\{
    \begin{array}{ll}
        A+JC & \text{if $n=2s$} \\
        \left(\begin{array}{cc}
            A+JC & \sqrt{2}x \\
            \sqrt{2}y & q
        \end{array}\right) & \text{if $n=2s+1$}.
    \end{array}\right.
\end{align*}
Though the above theorem reduces the centrosymmetric matrix to a block diagonal matrix, the block matrices are not independent. Therefore, we need to take care of the dependence structure in the course of the proof.

%%%%%%%%%%%%%%%%%%%%%%%%%%%%%%%%%%%%%%%%%%%%%%%%%%%%%%%%%%%%%%%%%%%%%%%%%%%
%%%%%%%%%%%%%%%%%%%%%%%%%%%%%   Circular Law  %%%%%%%%%%%%%%%%%%%%%%%%%%%%%
%%%%%%%%%%%%%%%%%%%%%%%%%%%%%%%%%%%%%%%%%%%%%%%%%%%%%%%%%%%%%%%%%%%%%%%%%%%

\section{Circular Law for Random Centrosymmetric Matrix}\label{sec:circular law proof}
Let us first state the general circular law.

\setcounter{Thm}{0}
\renewcommand\theThm{\arabic{section}.\arabic{Thm}}

\begin{Thm}[Circular Law, \cite{tao2008random}]\label{GenCL}
   Consider a random matrix $M = \frac{1}{\sqrt{n}}[x_{ij}]_{i,j=1}^{n},$ where $x_{ij}$s are i.i.d. random variables with mean zero and variance $1$. Then, as $n\to\infty$, the ESD of $M$ converges almost surely to the uniform distribution on the unit disc of $\mathbb{C}$ with density $\frac{1}{\pi} 1_{\{ z\in \mathbb{C} \;|\; |z|\leq1\}}$.
\end{Thm}

 Now, let us state the circular law for random centrosymmetric matrix.

\begin{Thm}[Circular Law for Random Centrosymmetric Matrix]
  Consider a random centrosymmetric matrix $M = \frac{1}{\sqrt{n}}[x_{ij}]_{i,j=1}^{n},$ where $x_{ij}$s are i.i.d. random variables for $i\leq j$ with mean zero and variance $1.$ Then, as $n\to\infty$, the ESD of $M$ converges almost surely to the uniform distribution on the unit disc of $\mathbb{C}$ with density $\frac{1}{\pi} 1_{\{ z\in \mathbb{C} \;|\; |z|\leq1\}}$.
\end{Thm}
 \begin{figure}[H]
     \centering
     \includegraphics[scale = 0.3]{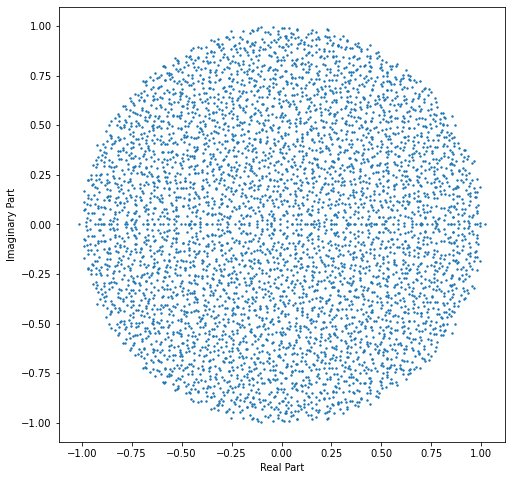}
     \caption{Scatter plot of
    real and imaginary parts of the eigenvalues of a $5000\times 5000$ random centrosymmetric matrix having standard Gaussian entries.}
     \label{fig:circle}
 \end{figure}

 \begin{proof}
     We have a centrosymmetric matrix $M = \frac{1}{\sqrt{n}}[x_{ij}]_{i,j=1}^{n},$ where $x_{ij}$s are i.i.d. centered random variables with variance $1$.
     
From Theorem \ref{orthogonal}, it is evident that the eigenvalues of $M$ are the eigenvalues of $T_{1}$ and $T_{2}.$ The entries of $T_{1}$ and $T_{2}$ are i.i.d. centered random variables having variance $2$. However, sizes of $T_{1}$ and $T_{2}$ is $\lceil n/2 \rceil\times \lceil n/2\rceil$, and they are scaled by $1/\sqrt{n}.$ Therefore, by Theorem \ref{GenCL} the ESD of  $T_{1}$ converges to the circular law. The same applies to $T_{2}$ as well. Hence, ESD of the matrix $M$ converges to the circular law.

 \end{proof}

%%%%%%%%%%%%%%%%%%%%%%%%%%%%%%%%%%%%%%%%%%%%%%%%%%%%%%%%%%%%%%%%%%%%%%%%%%%
%%%%%%%%%%%%%%%%%%%%%%%%%%%%%   CLT   PART  %%%%%%%%%%%%%%%%%%%%%%%%%%%%%%% 
%%%%%%%%%%%%%%%%%%%%%%%%%%%%%%%%%%%%%%%%%%%%%%%%%%%%%%%%%%%%%%%%%%%%%%%%%%%

\section{Proof of the central limit theorem}\label{sec: proof of the clt}

In Section \ref{sec:circular law proof}, we have seen that the limiting ESD follows the circular law. However, for any finite $n$, there might be some eigenvalues outside the unit disc. To avoid such situations, we shall work on the event where the eigenvalues are confined within a disc of finite radius. Keeping this in mind, we define the event
\begin{equation}
\Omega_{n} = \{\|T_1\| \leq \rho_1 \} \cap \{\|T_2\| \leq \rho_2 \} \label{eqn: event that T_1, T_2 are bounded},
\end{equation}
where $\rho_1$ and $\rho_2$ are same as $\rho,$ given in Lemma \ref{lemma:Norm bound}. Additionally, Lemma \ref{lemma:Norm bound} also implies that
\begin{equation}\label{eqn: probability bound on omega complement}
P\left(\Omega_n^{c}\right) \leq K \exp\left(-\sqrt{\frac{\alpha n}{2}} \frac{3 \rho_{m}}{4}\right), \quad \text{where } \rho_{m} = \min\{\rho_1, \rho_2\}.
\end{equation}
Let us denote $\rho=\max\{\rho_{1}, \rho_{2}\}.$ 
%If $f: \mathbb{C} \rightarrow \mathbb{C}$ is analytic on $\mathbb{D}_{\rho+\tau}$, then \textcolor{red}{on the event $\Omega_n$ we may write}
If $ f $ is a complex analytic function on $ \partial \mathbb{D}_{\rho+\tau} $, then we can express $\operatorname{L}_n(f) $ on the event $ \Omega_n $ as follows 

\begin{align}
\operatorname{L}_{f}(M)\notag
&= \sum_{i=1}^{n} f(\lambda_{i}(M)) \notag \\
&= \frac{1}{2\pi i} \oint_{\partial \mathbb{D}_{\rho+\tau}} \sum_{i=1}^{n} \frac{f(z)}{z-\lambda_{i}(M)}\, dz \notag \\
&= \frac{1}{2\pi i} \oint_{\partial \mathbb{D}_{\rho+\tau}} f(z) \Tr \mathcal{R}_{z}(M) \, dz. \notag
\end{align}
The above represents $\operatorname{L}_{f}(M)$ in terms of the resolvent of the matrix. This brings down our task to the analysis of the resolvent. However, to keep the norm of the resolvent bounded, we need to restrict our attention to the event $\Omega_{n}$, where norms of $T_{1}, T_{2}$ are bounded. From the probability estimate \eqref{eqn: probability bound on omega complement} and the Lemma \ref{lemma:Norm bound}, we may express the centered linear eigenvalue statistic as follows. 
$$
\operatorname{L}_{f}^{\circ}(M) = \frac{1}{2\pi i} \oint_{\partial \mathbb{D}_{\rho+\tau}} f(z) \Tr \hat{\mathcal{R}}_{z}^{\circ}(M) \, dz+o(1).
$$
where $\hat{\xi} = \xi1_{\Omega_{n}},$ and $\xi^{\circ} = \xi-\mathbb{E}[\xi].$ The $o(1)$ at then end of the above expression can further be explained as follows.

We can rewrite
$$\Tr{\mathcal{R}}_{z}^{\circ}(M)=\Tr \hat{\mathcal{R}}_{z}^{\circ}(M)+\Tr \mathcal{R}_{z}(M)1_{\Omega_{n}^c}-\mathbb{E}[\Tr \mathcal{R}_{z}(M)1_{\Omega_{n}^c}].$$

We will demonstrate that $\Tr \mathcal{R}_{z}(M)1_{\Omega_{n}^c}-\mathbb{E}[\Tr \mathcal{R}_{z}(M)1_{\Omega_{n}^c}]=o(1).$ First, let us prove that $\Tr \mathcal{R}_{z}(M)1_{\Omega_{n}^c}=o(1).$ Since we are dealing with continuous random variables, almost surely there is no eigenvalue on the circle $\partial\mathbb{D}_{\rho+\tau}.$ Therefore, $\sup_{z\in \partial\mathbb{D}_{\rho+\tau}}|\Tr \mathcal{R}_{z}(M)|$ is bounded almost surely. Consequently, we conclude

$$\mathbb{P}\left(\sup_{z\in \partial \mathbb{D}_{\rho+\tau}}|\Tr \mathcal{R}_{z}(M)1_{\Omega_{n}^c}|>\epsilon \right)\leq \mathbb{P}(\Omega_n^c)\rightarrow 0 \text{ as } n \rightarrow \infty.$$

Next, we prove that $\mathbb{E}[\Tr \mathcal{R}_{z}(M)1_{\Omega_{n}^c}] \rightarrow 0$ asymptotically.

We note that the assumptions listed in Condition \ref{cond:matrixcond} are consistent with the assumptions of the Lemma \ref{lemma: rider silverstein bound on trace of resolvent}. Therefore, using the Lemma \ref{lemma: rider silverstein bound on trace of resolvent}, and the bound \eqref{eqn: probability bound on omega complement} we may conclude that 
\begin{align*}
    \mathbb{E}[\Tr \mathcal{R}_{z}(M)1_{\Omega_{n}^c}]&\leq \mathbb{E}[|\Tr \mathcal{R}_{z}(M)|]\mathbb{P}(\Omega_{n}^{c})\\
    &\leq C n^{7/2}\exp\left(-\sqrt{\frac{\alpha n}{2}} \frac{3 \rho_{m}}{4}\right)\\
    &\to 0\;\;\;\textbf{as $n\to\infty$.}
\end{align*}

% \begin{align*}
%     \mathbb{E}[\Tr \mathcal{R}_{z}(M)1_{\Omega_{n}^c}]
%     &= \mathbb{E}\left[\frac{n}{z}1_{\Omega_{n}^c}+ \sum_{k=1}^{\infty}\frac{1}{z^{k+1}}\Tr(M^k)1_{\Omega_{n}^c}\right]\\
%     &= \frac{n}{z}\mathbb{P}(\Omega_{n}^c)+\sum_{k=1}^{\infty}\frac{1}{z^{k+1}}\mathbb{E}\left[\Tr M^k 1_{\Omega_{n}^c}\right]\\
%     &\leq \frac{n}{z}\mathbb{P}(\Omega_{n}^c)+\sum_{k=1}^{\infty}\frac{n}{z^{k+1}}[\mathbb{E}\|M\|^{2k}]^{1/2} [\mathbb{P}(\Omega_{n}^c)]^{1/2}\\
%     &\leq \frac{n}{z}\mathbb{P}(\Omega_{n}^c) + \sum_{k=1}^{\infty}\left(\frac{\rho^k+2}{z^{k+1}}\right) n^{-2} \rightarrow 0 \text{ as } n\rightarrow \infty.
% \end{align*}

% The last inequality follows from $\mathbb{E}[\|M\|^{2k}]\leq C (\rho^k+2)n^{-2}$ by using Lemma \ref{lemma:Norm bound}, where $C$ is universal constant.

Furthermore, standard use of complex analysis shows that 
\begin{align*}
    &\frac{1}{\pi}\int_{\mathbb{D}_{1}}f'(w)\overline{f'(w)}\;d\mathfrak{R}(w)d\mathfrak{I}(w)\\
    &=\frac{1}{\pi}\int_{\mathbb{D}_{1}}\left[\frac{1}{2\pi i}\oint_{\partial\mathbb{D}_{1}}\frac{f(z)\;dz}{(w-z)^{2}}\frac{1}{2\pi i}\oint_{\partial\mathbb{D}_{1}}\frac{f(\bar\eta)\;d\bar\eta}{(\bar w-\bar\eta)^{2}}\right]\;d\mathfrak{R}(w)d\mathfrak{I}(w)\\
    &=-\frac{1}{4\pi^{2}}\oint_{\partial\mathbb{D}_{1}}\oint_{\partial\mathbb{D}_{1}}f(z)\overline{f(\eta)}\left[\frac{1}{\pi}\int_{\mathbb{D}_{1}}\frac{d\mathfrak{R}(w)d\mathfrak{I}(w)}{(w-z)^{2}(\bar w-\bar\eta)^{2}}\right]\;dz\;d\bar\eta\\
    &=-\frac{1}{4\pi^{2}}\oint_{\partial\mathbb{D}_{1}}\oint_{\partial\mathbb{D}_{1}}\frac{f(z)\overline{f(\eta)}}{(1-z\bar\eta)^{2}}\;dz\;d\bar\eta.
\end{align*}
Hence, the main Theorem \ref{thm: main theorem} is essentially the same as proving Proposition \ref{prop: Gaussian process and tightness}.
\begin{prop}{\label{prop: Gaussian process and tightness}}
The sequence $\{\Tr \hat{\mathcal{R}}_{z}^{\circ}(M)\}_{n}$ is tight in the space $\mathcal{C}(\partial{\mathbb{D}}_{\rho+\tau})$, which is the space of continuous function on $\partial\mathbb{D}_{\rho+\tau},$ and converges in distribution to a Gaussian process with covariance kernel given by $2(1-z\bar \eta)^{-2}$.
\end{prop}

 Our first goal is to reduce the above proposition to \eqref{eq: first equivalent cond of MCLT}, \eqref{eq: second equivalent cond of MCLT}, \eqref{eq: third equivalent cond of MCLT}. First of all, by Cramer-Wold device, to show finite dimensional convergence it is sufficient to show that for any
    $\left\{z_l; 1 \leq l \leq q\right\} \subset \partial \mathbb{D}_{\rho + \tau},$ and $\left\{\alpha_l, \beta_l; 1 \leq l \leq q\right\} \subset \mathbb{C}$ for which
\begin{equation}
    \sum_{l=1}^q\left[\alpha_l \Tr  \hat{\mathcal{R}}_{z_l}^{\circ}(M)+\beta_l \overline{\Tr  \hat{\mathcal{R}}_{z_l}^{\circ}(M)}\right] \mathbf{1}_{\Omega_n}
\end{equation}
is real and converges to a Gaussian random variable with mean zero. In order to show this, we use the martingale CLT, which is stated in Lemma \ref{lemma:MCLT}. For notaional convenience, we define the relevant conditional expectations with respect to the columns of the matrices $T_{1}$ and $T_{2}.$  Let $\Xi_{k}[\xi]$ be the averaging of the random variable $\xi$ over the $k$-th columns of matrices $T_{1}$ and $T_{2}$, and $\mathbb{E}_k[\cdot]=\Xi_{k+1} \Xi_{k+2} \ldots \Xi_{\lceil \frac{n}{2} \rceil}[\cdot]=\Xi_{k+1} \mathbb{E}_{k+1}[\cdot]$. For instance, $\mathbb{E}_0[\cdot]=\mathbb{E}[\cdot]$ and $\mathbb{E}_{\lceil \frac{n}{2}\rceil} [\cdot]=[\cdot]$. Therefore, we can express
\begin{align}
    \Tr  \hat{\mathcal{R}}_z^{\circ}(M) \notag
    &= \Tr  \hat{\mathcal{R}}_z^{\circ}(T_1)+\Tr  \hat{\mathcal{R}}_z^{\circ}(T_2)\\
    &=\Tr \hat{\mathcal{R}}_{z}(T_1)-\mathbb{E}[\Tr \hat{\mathcal{R}}_{z}(T_1)]+\Tr \hat{\mathcal{R}}_{z}(T_2)-\mathbb{E}[\Tr \hat{\mathcal{R}}_{z}(T_2)] \notag\\
    &=   \sum_{k=1}^{\lceil n/2\rceil} (\mathbb{E}_{k}-\mathbb{E}_{k-1})\Tr \hat{\mathcal{R}}_{z}(T_1) +  \sum_{k=1}^{\lceil n/2\rceil} (\mathbb{E}_{k}-\mathbb{E}_{k-1})\Tr \hat{\mathcal{R}}_{z}(T_2) \notag\\
    &=\sum_{k=1}^{\lceil n/2\rceil} \mathbb{E}_k\left\{\left(\Tr  \hat{\mathcal{R}}_z(T_1)\right)_k^\circ\right\}+\sum_{k=1}^{\lceil n/2\rceil} \mathbb{E}_k\left\{\left(\Tr  \hat{\mathcal{R}}_z(T_2)\right)_k^\circ\right\}\notag\\
    &=: \sum_{k=1}^{\lceil n/2\rceil} W_{z, k}(T_1)+\sum_{k=1}^{\lceil n/2\rceil} W_{z, k}(T_2) \notag\\
    &= \sum_{k=1}^{\lceil n/2\rceil}\left(W_{z, k}(T_1)+W_{z, k}(T_2)\right)\label{eq: trace as martingale difference},
\end{align}
where $\xi_k^{\circ} = \xi-\Xi_{k}[\xi].$ Since $\mathbb{E}_{k}[\cdot]$ is defined as the conditional expectation given the first $k$ columns of $T_1$ and $T_2$, $(\mathbb{E}_k - \mathbb{E}_{k-1})[\xi]$ forms a martingale difference sequence. Consequently, the sequences $\left\{W_{z, k}(T_1)\right\}_{1 \leq k \leq \lceil n/2 \rceil}$ and $\left\{W_{z, k}(T_2)\right\}_{1 \leq k \leq \lceil n/2 \rceil}$ are martingale difference sequences with respect to the filtration $\mathcal{F}_{n, k} = \sigma\{t^{1}_{j}, t^{2}_{j} : 1 \leq j \leq k\}$, where $t^{1}_{j}$ and $t^{2}_{j}$ are the $j$th columns of matrices $T_{1}$ and $T_{2}$ respectively. Therefore, the sum $W_{z, k}(T_1)+W_{z, k}(T_2)$ also forms a martingale difference sequence with respect to the same filtration $\mathcal{F}_{n, k}.$

Now, we can rewrite
 \begin{align}
       &\sum_{i=1}^q\left[\alpha_i \Tr  \hat{\mathcal{R}}_{z_i}^{\circ}(M)+\beta_i \overline{\Tr  \hat{\mathcal{R}}_{z_i}^{\circ}(M)}\right]\notag\\
       &=\sum_{k=1}^{\lceil n/2\rceil} \left\{\sum_{i=1}^{q}\left(\alpha_{i} (W_{z_{i},k}(T_1)+W_{z_{i},k}(T_2))+ \beta_i (W_{\bar{z_{i}},k}(T_{1}\conjugate)+W_{\bar{z_{i}},k}(T_{2}\conjugate))\right)\right\}\notag\\
       &= \sum_{k=1}^{\lceil n/2\rceil} \left\{\sum_{i=1}^{q}\alpha_{i} W_{z_{i},k}(T_1)+ \beta_i W_{\bar{z_{i}},k}(T_{1}\conjugate) +\sum_{i=1}^{q}\alpha_{i} W_{z_{i},k}(T_2)+ \beta_i W_{\bar{z_{i}},k}(T_{2}\conjugate)\right\} \notag\\
       &=: \sum_{k=1}^{\lceil n/2\rceil} \left(\psi_{n,k}^{(1)}+\psi_{n,k}^{(2)}\right)\notag\\
       &=:\sum_{k=1}^{\lceil n/2\rceil} \psi_{n,k}\notag,
\end{align}
where $X\conjugate$ denotes the complex conjugate of the matrix $X$. 
As a result, $\left\{\psi_{n, k}\right\}_{1 \leq k \leq {\lceil n/2 \rceil}}$ constitutes a martingale difference sequence with respect to the filtration $\mathcal{F}_{n, k}$. We notice that proving the following equation \eqref{eq: first equivalent cond of MCLT} is sufficient for the validity of the condition (i) of Lemma \ref{lemma:MCLT};

\begin{align}
    \sum_{k=1}^{\lceil n/2 \rceil} \mathbb{E}\left[\psi_{n, k}^2 \mathbf{1}_{\left|\psi_{n, k}\right|>\theta}\right] \leq \frac{1}{ \theta^2} \sum_{k=1}^{\lceil n/2 \rceil} \mathbb{E}\left[\psi_{n, k}^4\right]  \rightarrow 0 \label{eq: first equivalent cond of MCLT}.
\end{align}
Since $\mathbb{E}[\cdot|\mathcal{F}_{n,k-1}]= \mathbb{E}_{k-1}[\cdot],$ the condition (ii) of Lemma \ref{lemma:MCLT} can be expressed as

\begin{align}
\sum_{k=1}^{\lceil n/2 \rceil}\mathbb{E}_{k-1}[(W_{z_i,k}(T_1)+W_{z_i,k}(T_2))(W_{z_j,k}(T_1)+W_{z_j,k}(T_2))]\stackrel{p}{\rightarrow} 0, \label{eq: second equivalent cond of MCLT} 
\end{align}
and
\begin{align}
\mathbb{E}_{k-1}[(W_{z_{i},k}(T_1)+W_{z_{i},k}(T_2))(W_{\bar{z_{j}},k}(T_1\conjugate)+W_{\bar{z_{j}},k}(T_2\conjugate))] \stackrel{p}{\rightarrow} \sigma\left(z_i, z_j\right), \label{eq: third equivalent cond of MCLT}
\end{align}
where $\sigma$ is a covariance kernel to be found later. We defer the calculation of the covariance kernel to section \ref{sec: calculation of the covariance kernel}. Hence, establishing Proposition \ref{prop: Gaussian process and tightness} is equivalent to demonstrating the tightness of the sequence $\{\Tr  \hat{\mathcal{R}}_{z}^{\circ}(T_1)+\Tr \hat{\mathcal{R}}_{z}^{\circ}(T_2)\}_n$ and verifying the conditions outlined in \eqref{eq: first equivalent cond of MCLT}, \eqref{eq: second equivalent cond of MCLT}, and \eqref{eq: third equivalent cond of MCLT}. Before proceeding any further, let us do some reductions.

For $i=1,2$, let $T_{i}^{(k)}$ be the matrix obtained by zeroing out the $k$th column of $T_i$, and let $\hat{\mathcal{R}}_{z}(T_i^k) := \mathcal{R}_{z}(T_i^k)\mathbf{1}_{\Omega_{n,k}}$,
where $\Omega_{n,k}=  \left\{\norm{T_{1}^{(k)}}\leq \rho_1 \right\} \cap \left\{\norm{T_{2}^{(k)}}\leq \rho_2 \right\}.$ The space $\Omega_{n,k}$ is different from $\Omega_{n},$ but asymptotically both the spaces are the same. In fact, we can check that
\begin{align*}
    &\mathbb{E}\left|\sum_{i=1}^{\lceil n/2 \rceil}\Xi_{i}\left[\Tr \mathcal{R}_{z}(T_{1}^{(k)})\mathbf{1}_{\Omega_{n}}-\Tr \mathcal{R}_{z}(T_{1}^{(k)})\mathbf{1}_{\Omega_{n, i}}\right]_{i}^{\circ}\right|\\
    &\leq 2\mathbb{E}\left|\sum_{i=1}^{\lceil n/2 \rceil}\Xi_{i}\left[\Tr \mathcal{R}_{z}(T_{1}^{(k)})\mathbf{1}_{\Omega_{n}^{c}\cap \Omega_{n, i}}\right]\right|\\
    &\leq \frac{n^{2}}{\tau}\mathbb{P}(\Omega_{n}^c)=o(1).
\end{align*}
The same is true for $T_{2}^{(k)}$ as well. The important difference between $\Omega_{n}$ and $\Omega_{n, k}$ is that the latter is independent of the $k$-th columns of $T_{1}$ and $T_{2}$, which will allow us to decouple some expectation calculations in future. Consequently, we shall proceed with the space $\Omega_{n,k}$.

Using Lemma \ref{lem: sherman morrison} and the resolvent identity, we have
\begin{align*}
    \Tr \hat{\mathcal{R}}_z(T_1) &=\Tr \hat{\mathcal{R}}_z(T_1^k)+ e_{k}^{t}\hat{\mathcal{R}}_z(T_1^k) \mathcal{R}_{z}(T_1)t_k^1\\
    &=\Tr \hat{\mathcal{R}}_z(T_1^k)+\frac{e_k^t(\hat{\mathcal{R}}_z(T_1^k))^2t_k^1}{1+e_k^t \hat{\mathcal{R}}_z(T_1^k)t_k^1}.
\end{align*}

Similarly, 
$$
\Tr \hat{\mathcal{R}}_z(T_2) = \Tr \hat{\mathcal{R}}_z(T_2^k) + \frac{e_k^t(\hat{\mathcal{R}}_z(T_2^k))^2t_k^2}{1 + e_k^t \hat{\mathcal{R}}_z(T_2^k)t_k^2},
$$
where $t_k^1$ and $t_k^2$ are the $k$th columns of $T_1$ and $T_2$ respectively. Thus,
\begin{align}
    &\Tr \hat{\mathcal{R}}_z(M)\notag\\ &= \Tr \hat{\mathcal{R}}_z(T_1^k)+\Tr \hat{\mathcal{R}}_z(T_2^k)+\frac{e_k^t(\hat{\mathcal{R}}_z(T_1^k))^2t_k^1}{1+e_k^t \hat{\mathcal{R}}_z(T_1^k)t_k^1}+\frac{e_k^t(\hat{\mathcal{R}}_z(T_2^k))^2t_k^2}{1+e_k^t \hat{\mathcal{R}}_z(T_2^k)t_k^2}\notag\\
    &=\Tr \hat{\mathcal{R}}_z(T_1^k)+\Tr \hat{\mathcal{R}}_z(T_2^k) -\frac{\partial}{\partial z} \log(1+\gamma_{1k}(z))-\frac{\partial}{\partial z} \log(1+\gamma_{2k}(z))\notag,
\end{align}
where $\gamma_{1k}(z):=e_k^t\hat{\mathcal{R}}_z(T_1^k)t_k^1$
and $ \gamma_{2k}(z):=e_k^t\hat{\mathcal{R}}_z(T_2^k)t_k^2$. Now, let us find the tail bounds of $\gamma_{1k}(z)$ and $\gamma_{2k}(z).$ Note that $\gamma_{1k}(z)$ is product of two independent random variables. Intuitively, $\gamma_{1k}(z)\stackrel{a.s.}{\to} 0$ by conditioning on $\{t^{1}_{j}, t^{2}_{j} : j\in \{1,2,...,\lceil n/2 \rceil \} \backslash \{k\}\}$.

Recall that $\sum_{j=1}^{\lceil n/2 \rceil} |\hat{\mathcal{R}_{z}}(T_{1}^k)_{ij}|\leq \|\hat{\mathcal{R}_{z}}(T_{1}^k)\|^2 \leq \tau^{-2}$ for $z\in \partial{\mathbb{D}}_{\rho + \tau}.$ Now, treating $\gamma_{1k}(z)$ as the function $g$ in the property 3 in the Definition \ref{def:Poincaré inequality}, we obtain

%Define $h(t_k^1) = e_k^t\hat{\mathcal{R}}_z(T_1^k)t_k^1= \gamma_{1k}(z).$ Then by using property $3$ in Definition \ref{def:Poincaré inequality} and the fact that $\sum_{j=1}^{\lceil n/2 \rceil} |\hat{\mathcal{R}_{z}}(T_{1}^k)_{ij}|\leq \|\hat{\mathcal{R}_{z}}(T_{1}^k)\|^2 \leq \tau^{-2}$ for $z\in \partial{\mathbb{D}}_{\rho + \tau}$, we have
\begin{equation}
    \mathbb{P}\left(|\gamma_{1k}(z)|>t \;\;|\;\; \hat{\mathcal{R}_{z}}(T_{1}^k)\right) \leq K \operatorname{exp} \left(-\tau \sqrt{\frac{\alpha n}{2}}t\right) \label{eq: probability estimate of gamma_1k}.
\end{equation}
In addition,

\begin{equation}
    \mathbb{E}[|\gamma_{1k}(z)|^p] \leq K p! \left(\frac{2}{\tau^2 \alpha n}\right)^{p/2} \text{ for all } z \in \partial{\mathbb{D}}_{\rho+\tau}.
    \label{eq: moment of gamma_1k}
\end{equation}

Similarly, we have bound for $\gamma_{2k}$ as follows.
\begin{equation}
    \mathbb{E}[|\gamma_{2k}(z)|^p] \leq K p! \left(\frac{2}{\tau^2 \alpha n}\right)^{p/2} \text{ for all } z \in \partial{\mathbb{D}}_{\rho+\tau}
    \label{eq: moment of gamma_2k}.
\end{equation}

Since $T_{1}^{(k)}, T_{2}^{(k)}$ are independent of the $k$th column of $T_{1}, T_{2}$ respectively, we can easily see that 
\begin{align}
   &\left[\Tr \hat{\mathcal{R}_{z}}(T_1^k)+\Tr \hat{\mathcal{R}_{z}}(T_2^k)\right]_k^\circ \nonumber\\
   &=\Tr \hat{\mathcal{R}_{z}}(T_1^k)+\Tr \hat{\mathcal{R}_{z}}(T_2^k)- \Xi_{k}[\Tr \hat{\mathcal{R}_{z}}(T_1^k)+\Tr \hat{\mathcal{R}_{z}}(T_2^k)] =0 \label{eqn: centering of R after removing kth column}.
\end{align}

Hence, 
\begin{align}
    &W_{z, k}(T_1)+W_{z, k}(T_2)\\
    &= \mathbb{E}_k\left[\left(\Tr \hat{\mathcal{R}}_{z}(T_1)+\Tr \hat{\mathcal{R}}_{z}(T_2)\right)_k^\circ\right]\notag\\
    &= - \mathbb{E}_k\left[\left(\frac{\partial}{\partial z} \operatorname{\log}(1+\gamma_{1k}(z)) + \frac{\partial}{\partial z} \operatorname{\log}(1+\gamma_{2k}(z)) \right)_k^\circ \right]\notag\\
    &= - \frac{\partial}{\partial z} \mathbb{E}_k[\operatorname{\log}(1+\gamma_{1k}(z))+\operatorname{\log}(1+\gamma_{2k}(z))]_k^\circ \notag\\
    &= - \frac{\partial}{\partial z}\mathbb{E}_k
       [\operatorname{\log}\{(1+\gamma_{1k}(z))(1+\gamma_{2k}(z))\}]_k^\circ\label{eq: W_z_k as derivative of log}.
\end{align} 

In the above, $\mathbb{E}_{k}$ and $\partial/\partial z$ are interchanged using the dominated convergence theorem. We can estimate $W_{z,k}(T_1)+W_{z,k}(T_2)$ by the following fact from complex analysis.
Let $v$ be an analytic function on $z+\mathbb{D}_{\tau/2}.$ Then
 \begin{equation}
     |v'(z)| = \left|\frac{1}{2\pi i} \oint_{z+\partial{\mathbb{D}_{\tau/3}}} \frac{v(t)dt}{(t-z)^2}\right| \leq \frac{5}{\pi \tau^2} \oint_{z+\partial{\mathbb{D}_{\tau/3}}}|v(t)|dt\label{eq: Bound on derivative of an analytic function}.
 \end{equation}

\subsection{Proof of tightness}
 In this subsection, we show the tightness of $\{\Tr \hat{\mathcal{R}}_{z}^{\circ}(M)\}_n$ in the space $\mathcal{C}(\partial{\mathbb{D}}_{\rho+\tau})$.
Since
$\Tr \hat{\mathcal{R}}_z^{\circ}(M) = \Tr \hat{\mathcal{R}}_z^{\circ}(T_1)+\Tr \hat{\mathcal{R}}_z^{\circ}(T_2)$, it is sufficient to show that both $\{\Tr \hat{\mathcal{R}}_z^{\circ}(T_1)\}_{n}$ and $\{\Tr \hat{\mathcal{R}}_z^{\circ}(T_2)\}_{n}$ are tight. To prove the tightness of $\Tr \hat{\mathcal{R}}_z^{\circ}(T_1)$ we need to show that for any $\epsilon>0$ there exit a compact set $\mathcal{K}(\epsilon)\subset \mathcal{C}(\partial{\mathbb{D}}_{\rho+\tau}) $ such that
\begin{equation}
    \mathbb{P}\left(\Tr \hat{\mathcal{R}}_z^{\circ}(T_1)\in \mathcal{K}(\epsilon)^c\right)< \epsilon. \label{eq: tightness main probability condition}
\end{equation}

However by Arzela-Ascoli theorem, the compact sets in space of continuous functions are space of equicontinuous functions. Let us define,
$$\Phi_n:= \{|\gamma_{1k}(z)|<n^{-1/8}, \text{  for all  } 1\leq k\leq \lceil n/2 \rceil \}.$$

Therefore from \eqref{eq: probability estimate of gamma_1k}, 
$$
\mathbb{P}(\Phi_n^c) \leq K  \exp{\left(\log\left(\frac{n}{2}\right)-\tau \sqrt{\frac{\alpha}{2 }} n^{\frac{3}{8}} \right)} = o(1).
$$

 Consequently,
 \begin{align}
     \mathbb{P}\left(\Tr \hat{\mathcal{R}}_z^{\circ}(T_1)\in \mathcal{K}(\epsilon)^c\right) \notag
     &\leq \mathbb{P}\left(  \frac{| \Tr \hat{\mathcal{R}}_z^{\circ}(T_1)-\Tr \hat{\mathcal{R}}_\eta^{\circ}(T_1)|}{|z-\eta|} > K \right) \notag\\
     &\leq \mathbb{P}\left(   \frac{| \Tr \hat{\mathcal{R}}_z^{\circ}(T_1)-\Tr \hat{\mathcal{R}}_\eta^{\circ}(T_1)|}{|z-\eta|} 1_{\Phi_n} > K \right) + o(1),
 \end{align}
for some $K>0$. Now, the following equation \eqref{eq: expectation of square of difference of resolvents} is sufficient to prove \eqref{eq: tightness main probability condition};
\begin{equation}
    \mathbb{E}\left\{\left|\frac{|\Tr \hat{\mathcal{R}}_z^{\circ}(T_1)-\Tr \hat{\mathcal{R}}_\eta^{\circ}(T_1)}{z-\eta} \right|^2  1_{\Phi_n}\right\} \leq C, \label{eq: expectation of square of difference of resolvents}
\end{equation}
uniformly for all $z, \eta \in \partial{\mathbb{D}}_{\rho+\tau}$ and $n\in \mathbb{N}$. Using the resolvent identity, we can write
$$
\Tr  \hat{\mathcal{R}}_z(T_{1})-\Tr  \hat{\mathcal{R}}_\eta(T_{1})=(z-\eta) \Tr \left[\hat{\mathcal{R}}_z(T_{1})\hat{\mathcal{R}}_\eta(T_{1})\right].
$$
Therefore, using \eqref{eq: trace as martingale difference} we have
\begin{align}
    (z-\eta)^{-1}\left[\Tr  \hat{\mathcal{R}}_z^\circ(T_{1})-\Tr  \hat{\mathcal{R}}_\eta^\circ(T_{1})\right]=\sum_{k=1}^{\lceil n/2 \rceil} \mathbb{E}_k\left\{ \left(\Tr \left[\hat{\mathcal{R}}_k(T_{1}) \hat{\mathcal{R}}_\eta(T_{1})\right)\right]_k^\circ \right\}. \label{eqn: tightness: difference of tr R as a product of R}
\end{align}

Using Lemma \ref{lem: sherman morrison} and the resolvent identity, we have
\begin{align*}
\Tr &\left[\hat{\mathcal{R}}_z(T_1) \hat{\mathcal{R}}_\eta(T_1)-\hat{\mathcal{R}}_z\left(T_{1}^{(k)}\right) \hat{\mathcal{R}}_\eta\left(T_{1}^{(k)}\right)\right] \\
&=\Tr \left[\left\{\hat{\mathcal{R}}_z(T_1)-\hat{\mathcal{R}}_z(T_{1}^{(k)})\right\} \left\{\hat{\mathcal{R}}_\eta(T_1)-\hat{\mathcal{R}}_\eta(T_{1}^{(k)})\right\}\right] \\
&\quad+\Tr \left[\hat{\mathcal{R}}_z\left(T_{1}^{(k)}\right)\left\{\hat{\mathcal{R}}_\eta(T_1)-\hat{\mathcal{R}}_\eta\left(T_{1}^{(k)}\right)\right\}\right]+\Tr \left[\left\{\hat{\mathcal{R}}_z(T_1)-\hat{\mathcal{R}}_z\left(T_{1}^{(k)}\right)\right\} \hat{\mathcal{R}}_\eta\left(T_{1}^{(k)}\right)\right] \\
&=\Tr  \frac{\hat{\mathcal{R}}_z\left(T_{1}^{(k)}\right) t_k^1 e_k^t \hat{\mathcal{R}}_z\left(T_{1}^{(k)}\right)}{1+\gamma_{k1}(z)} \frac{\hat{\mathcal{R}}_\eta\left(T_{1}^{(k)}\right) t_k^1 e_k^t \hat{\mathcal{R}}_\eta\left(T_{1}^{(k)}\right)}{1+\gamma_{k1}(\eta)} \\
&\quad+\Tr  \frac{\hat{\mathcal{R}}_z\left(T_{1}^{(k)}\right) \hat{\mathcal{R}}_\eta \left(T_{1}^{(k)}\right) t_k^1 e_k^t \hat{\mathcal{R}}_\eta \left(T_{1}^{(k)}\right)}{1+\gamma_{k1}(\eta)}+\Tr \frac{\hat{\mathcal{R}}_z\left(T_{1}^{(k)}\right) t_k^1 e_k^t \hat{\mathcal{R}}_z\left(T_{1}^{(k)}\right) \hat{\mathcal{R}}_\eta\left(T_{1}^{(k)}\right)}{1+\gamma_{k1}(z)} \\
&=\frac{\left(e_k^t \hat{\mathcal{R}}_z\left(T_{1}^{(k)}\right) \hat{\mathcal{R}}_\eta\left(T_{1}^{(k)}\right) t_k^1\right)^2}{\left(1+\gamma_{k1}(z)\right)\left(1+\gamma_{k1}(\eta)\right)}+\frac{e_k^t \hat{\mathcal{R}}_\eta\left(T_{1}^{(k)}\right) \hat{\mathcal{R}}_z\left(T_{1}^{(k)}\right) \hat{\mathcal{R}}_\eta\left(T_{1}^{(k)}\right) t_k^1}{1+\gamma_{k1}(\eta)} \\
&\quad+\frac{e_k^{t} \hat{\mathcal{R}}_z\left(T_{1}^{(k)}\right) \hat{\mathcal{R}}_\eta \left(T_{1}^{(k)}\right) \hat{\mathcal{R}}_z\left(T_{1}^{(k)}\right) t_k^1}{1+\gamma_{k1}(z)} \\
&=:\vartheta_1(k)+\vartheta_2(k)+\vartheta_3(k).
\end{align*}

By \eqref{eqn: centering of R after removing kth column}, we have $\left[\Tr  \left(\hat{\mathcal{R}}_z\left(T_{1}^{(k)}\right) \hat{\mathcal{R}}_\eta\left(T_{1}^{(k)}\right)\right)\right]_k^\circ=0$. Therefore, we can express $\eqref{eqn: tightness: difference of tr R as a product of R}$ as

$$
(z-\eta)^{-1}\left[\Tr  \hat{\mathcal{R}}_z(T_1)-\Tr  \hat{\mathcal{R}}_\eta^{\circ}(T_1)\right]=\sum_{k=1}^{\lceil n/2 \rceil} \mathbb{E}_k\left\{\left[\vartheta_1(k)+\vartheta_2(k)+\vartheta_3(k)\right]_k^\circ \right\} .
$$
Since $\left\{\left[\vartheta_i(k)\right]_k^\circ \right\}$ is a martingale difference sequence,
$$
 \mathbb{E} \, \left| \, \sum_{k=1}^{\lceil n/2 \rceil} \mathbb{E}_k \left\{ \left[ \vartheta_{1}(k) + \vartheta_2(k) + \vartheta_3(k) \right]_k^\circ \right\} \right|^2 = \sum_{k=1}^{\lceil n/2 \rceil} \mathbb{E} \left| \mathbb{E}_k \left\{ \left[ \vartheta_1(k) + \vartheta_2(k) + \vartheta_3(k) \right]_k^\circ \right\} \right|^2.
$$
Hence, proving \eqref{eq: expectation of square of difference of resolvents} is equivalent to verifying the validity of
\begin{align}
    n\mathbb{E} [|\vartheta_i(k)|^2 1_{\mathrm{\Phi_n}}] \leq C, \text{ for all } 1 \leq k \leq \lceil n/2 \rceil, \;\; i=1,2,\dots. \label{eqn: tightness: theta^2 on Phi_n is bounded}
\end{align}
uniformly for all $n \in \mathbb{N}$ and $z, \eta \in \partial{\mathbb{D}}_{\rho + \tau}$. By adopting the same approach as demonstrated in \eqref{eq: moment of gamma_1k}, similar tail estimates for $e_k^t\hat{\mathcal{R}}_z\left(T_{1}^{(k)}\right) \hat{\mathcal{R}}_\eta \left(T_{1}^{(k)}\right) t_k^1$ can be obtained possibly with $\tau^{2}$ or $\tau^{3}$ in the rhs of \eqref{eq: moment of gamma_1k}. Noting that $|\gamma_{1k}(z)| < n^{-1/8}$ on $\Phi_n$, leveraging the estimate $|(1+\gamma_{1k}(z))|^{-1}\leq 2$ in $\mathbb{E}[|\vartheta_i(k)|^2],$ we have \eqref{eqn: tightness: theta^2 on Phi_n is bounded}. Similarly, we can prove the tightness of $\Tr \hat{\mathcal{R}}_z^{\circ}(T_2)$.

\subsection{Proof of \eqref{eq: first equivalent cond of MCLT}}

To estimate the right hand side of \eqref{eq: first equivalent cond of MCLT}, we first estimate the term $|W_{z,k}(T_1)+W_{z,k}(T_2)|^4$. We proceed by expanding $\operatorname{\log}(1+\gamma_{k1}(z))$ and $\operatorname{\log}(1+\gamma_{k2}(z))$ up to two terms first, and then using the estimates from \eqref{eq: moment of gamma_1k},\eqref{eq: moment of gamma_2k},\eqref{eq: W_z_k as derivative of log}, and \eqref{eq: Bound on derivative of an analytic function} we get the bound

    $$
    \mathbb{E}[|W_{z,k}(T_1)+W_{z,k}(T_2)|^4] = O(n^{-2}).
    $$

By substituting the above into \eqref{eq: first equivalent cond of MCLT}, we get
\begin{align*}
    \sum _{k=1}^{\lceil n/2 \rceil}\mathbb{E}[\psi_{n,k}^2 1_{|\psi_{n,k}|>\theta}] &\leq \frac{(2q)^3}{\theta^2}\sum_{k=1}^{\lceil n/2\rceil}\sum_{i=1}^{q}(|\alpha_i|^4 \mathbb{E}[|W_{z_{i},k}(T_1)+W_{z_{i},k}(T_2)|^4]\\
   &+|\beta_i|^4\mathbb{E}[|W_{\bar z_{i},k}(T_1\conjugate )+W_{\bar z_{i},k}(T_2\conjugate)|^4]) \\
   &= \theta^{-2}O(n^{-1}),
  \end{align*} 

which implies \eqref{eq: first equivalent cond of MCLT}.
\subsection{Proof of \eqref{eq: second equivalent cond of MCLT}}

Likewise, we can expand $\operatorname{\log}(1+\gamma_{k1}(z))$ and $\operatorname{\log}(1+\gamma_{k2}(z))$ up to two terms and use \eqref{eq: moment of gamma_1k}, \eqref{eq: moment of gamma_2k} to get the following estimate
$$
\Xi_{k}\{\mathbb{E}_k[(\log(1+\gamma_{1k}(z))+\log(1+\gamma_{2k}(z)))_k^\circ)]\mathbb{E}_k[(\log(1+\gamma_{1k}(\eta))+\log(1+\gamma_{2k}(\eta)))_k^\circ)]\} = O(n^{-2}).
$$
Thus, using \eqref{eq: W_z_k as derivative of log},\eqref{eq: Bound on derivative of an analytic function} and the above we have,
$$
\sum_{k=1}^{\lceil n/2 \rceil}\mathbb{E}_{k-1}[(W_{z_i,k}(T_1)+W_{z_i,k}(T_2))(W_{z_j,k}(T_1)+W_{z_j,k}(T_2))] = O(n^{-1}),
$$
which implies \eqref{eq: second equivalent cond of MCLT}.

\subsection{Proof of \eqref{eq: third equivalent cond of MCLT}}

Expanding $\operatorname{\log}(1+\gamma_{k1}(z))$ and $\operatorname{\log}(1+\gamma_{k2}(z))$ up to two terms and using \eqref{eq: moment of gamma_1k},\eqref{eq: moment of gamma_2k} and \eqref{eq: W_z_k as derivative of log} we have
$$
\mathbb{E}_{k-1}[(W_{z,k}(T_1)+W_{z,k}(T_2))(W_{\bar{\eta},k}(T_1\conjugate)+W_{\bar{\eta},k}(T_2\conjugate))] = \frac{\partial^2}{\partial z \partial \bar{\eta}}\Theta_{k}(z,\bar{\eta}),
$$
where 
\begin{align}
    &\Theta_{k}(z,\bar{\eta})\notag\\
    &= \Xi_{k}\{\mathbb{E}_k[(\log(1+\gamma_{1k}(z))+\log(1+\gamma_{2k}(z)))_k^\circ]\mathbb{E}_k[(\log(1+\overline{\gamma_{1k}(\eta)})+\log(1+\overline{\gamma_{2k}(\eta)}))_k^\circ]\}\notag\\
    &= \frac{1}{n}\left(e_k^t \mathbb{E}_k [\hat{\mathcal{R}}_{z}(T_1^k)]\Var((T_{1}^{(k)})_{ki})I\mathbb{E}_k [\hat{\mathcal{R}}_{\bar{\eta}}(T_1^k\conjugate)e_k]\right)\notag\\
    &+\frac{1}{n}\left(e_k^t \mathbb{E}_k [\hat{\mathcal{R}}_{z}(T_2^k)]\Var((T_{2}^{(k)})_{ki})I\mathbb{E}_k [\hat{\mathcal{R}}_{\bar{\eta}}(T_2^k\conjugate)e_k]\right)+O(n^{-2})\notag\\
    &=: \frac{1}{n}\Gamma_{k}(z,\bar{\eta})+O(n^{-2})\notag,
\end{align}
where $(T_{1}^{(k)})_{ij}$ and $(T_{2}^{(k)})_{ij}$ are the $ij$th element of $T_{1}^{(k)}$ and $T_{2}^{(k)}$ respectively.

% where $a_{ki} \text{ and } c_{ki} \text{ are the } (k,i)th \text{ element of the matrices } A \text{ and } C \text{ respectively.}$

Consequently, \eqref{eq: third equivalent cond of MCLT} becomes
\begin{align}
    &\sum_{k=1}^{\lceil n/2 \rceil}\mathbb{E}_{k-1}[(W_{z,k}(T_1)+W_{z,k}(T_2))(W_{\bar{\eta},k}(T_1)+W_{\bar{\eta},k}(T_2))]\notag\\
    &= \frac{\partial^2}{\partial z \partial \bar{\eta}}\left[\frac{1}{n}\sum _{k=1}^{\lceil n/2 \rceil}\Gamma_{k}(z,\bar{\eta})\right]+O(n^{-1}) \notag\\
    &=: \frac{\partial^2}{\partial z \partial \bar{\eta}}\chi_{k}(z,\bar{\eta})+O(n^{-1})\label{eqn: definition of chi}.
\end{align}

% \textcolor{orange}{Since $\|\hat{\mathcal{R}}_z(T_1^k)+\hat{\mathcal{R}}_z(T_2^k)\|\leq 4\tau^{-1}$ on $\mathbb{D}_{\rho+\tau/2}^c$, $\chi_{k}(z,\bar{\eta})$ is a sequence of uniformly bounded analytic functions on $\mathbb{D}_{\rho+\tau / 2}^c$. Therefore, by Vitali's theorem (see \cite[Section 2.9]{conway2012functions}), proving \eqref{eq: third equivalent cond of MCLT} equivalent to showing that $\chi_{k}(z,\bar{\eta})$ converges in probability.}
Since $\|\hat{\mathcal{R}}_z(T_1^k) + \hat{\mathcal{R}}_z(T_2^k)\| \leq 4\tau^{-1}$ on $\mathbb{D}_{\rho + \tau/2}^c$, the sequence $\chi_{k}(z, \bar{\eta})$ consists of uniformly bounded analytic functions on $\mathbb{D}_{\rho + \tau/2}^c$. Therefore, by Vitali's theorem (see \cite[Section 2.9]{conway2012functions}), proving \eqref{eq: third equivalent cond of MCLT} is equivalent to show that $\chi_{k}(z, \bar{\eta})$ converges in probability.

\begin{lemma}
    Let $\chi_{k}(z, \bar{\eta})$ be the same as defined in \eqref{eqn: definition of chi}, that is, $\chi_{k}(z, \bar{\eta}) = \frac{1}{n}\sum _{k=1}^{\lceil n/2 \rceil}\Gamma_{k}(z,\bar{\eta}).$ Then $\Var(\chi_{k}(z, \bar\eta))\to 0$ as $n\to\infty.$
\end{lemma}

\begin{proof}
Since $\{W_{z,k}(T_1)+W_{z,k}(T_2)\}_n$ is a martingale difference sequence, we have 
$$ 
\mathbb{E}[(W_{z,k}(T_1)+W_{z,k}(T_2))(W_{\bar{\eta},l}(T_1\conjugate)+W_{\bar{\eta},l}(T_2\conjugate))]=0 \text{ if } k \neq l. $$
This implies,
\begin{align}
    &\mathbb{E}\left\{\sum_{k=1}^{\lceil n/2 \rceil}\mathbb{E}_{k-1}[(W_{z,k}(T_1)+W_{z,k}(T_2))(W_{\bar{\eta},k}(T_1\conjugate)+W_{\bar{\eta},k}(T_2\conjugate))]\right\}\notag\\
    &= \mathbb{E}\left\{\left(\sum_{k=1}^{\lceil n/2 \rceil}(W_{z,k}(T_1)+W_{z,k}(T_2))\right)\left(\sum_{k=1}^{\lceil n/2 \rceil}(W_{\bar{\eta},k}(T_1\conjugate)+W_{\bar{\eta},k}(T_2\conjugate))\right)\right\}\notag\\
    &= \mathbb{E}\left\{\left(\Tr \hat{\mathcal{R}}_z^{\circ}(T_1)+\Tr \hat{\mathcal{R}}_z^{\circ}(T_2)\right)\left(\Tr \hat{\mathcal{R}}_\eta^{\circ}(T_1\conjugate)+\Tr \hat{\mathcal{R}}_\eta^{\circ}(T_2\conjugate)\right)\right\}\notag\\
    &= \mathbb{E}\left\{\Tr \hat{\mathcal{R}}_z^{\circ}(M) \Tr \hat{\mathcal{R}}_{\bar{\eta}}^{\circ}(M^*)\right\}.\notag
\end{align}
Let us rewrite $\Gamma_{k}(z,\bar{\eta})$ as
\begin{align}
    &\Gamma_{k}(z,\bar{\eta})\notag\\
    &=4e_k^t \mathbb{E}_k [\hat{\mathcal{R}}_{z}(T_1^k)]\mathbb{E}_k [\hat{\mathcal{R}}_{\bar{\eta}}(T_1^k\conjugate)]e_k+4e_k^t \mathbb{E}_k [\hat{\mathcal{R}}_{z}(T_2^k)]\mathbb{E}_k [\hat{\mathcal{R}}_{\bar{\eta}}(T_2^k\conjugate)]e_k\notag\\
    &:= \Gamma_{k}^1(z,\bar{\eta})+\Gamma_{k}^2(z,\bar{\eta}).\notag
\end{align}
Consequently, we may decompose
\begin{align}
    &\Var\left(\frac{1}{n}\Gamma_{k}(z,\bar{\eta})\right)\notag\\
    &=\Var\left[\frac{1}{n}\Gamma_{k}^1(z,\bar{\eta})+\frac{1}{n}\Gamma_{k}^2(z,\bar{\eta})\right] \notag\\
    &\leq 2\Var\left(\frac{1}{n}\Gamma_{k}^1(z,\bar{\eta})\right)+2\Var\left(\frac{1}{n}\Gamma_{k}^2(z,\bar{\eta})\right). \notag
\end{align}
Therefore, proving the lemma is equivalent to showing
\begin{align*}
    \Var\left(\frac{1}{n}\Gamma_{k}^1(z,\bar{\eta})\right)+\Var\left(\frac{1}{n}\Gamma_{k}^2(z,\bar{\eta})\right) \rightarrow 0.
\end{align*}
We show that $\Var\left(\frac{1}{n}\Gamma_{k}^1(z,\bar{\eta})\right) \rightarrow 0.$ We shall use the Poincar\'e inequality to get estimates on $\Var\left(\frac{1}{n}\Gamma_{k}^1(z,\bar{\eta})\right).$ This is shown in the equation \eqref{eqn: variance bound by poincare inequality}. However, since we are working on the event $\Omega_{n, k}$, where norms of the matrices are bounded, some terms may not be differentiable with respect to the matrix entries. To circumvent such a problem, we define a smoothing function as follows.

% Recall 
% $$\Gamma_{k}^1(z,\bar{\eta}) =4e_k^t \mathbb{E}_k [\mathcal{R}_{z}(T_1^k)]I\mathbb{E}_k [\mathcal{R}_{\bar{\eta}}(T_1^k\conjugate)]e_k. $$

Let $\epsilon_n= 4/n, \text{ and } \omega_{\epsilon_n,k}:\mathbb{R}^{\lceil n^2/4 \rceil} \rightarrow [0,1]$ be a smooth function such that 
\begin{align*}
&\omega_{\epsilon_n,k}|_{\{\|T_1^k\|\leq \rho_1\}} \equiv 1, \\
&\omega_{\epsilon_n,k}|_{\{\|T_1^k\|\geq \rho_1+\epsilon_n\}} \equiv 0,\\
&\left|\frac{\partial}{\partial x_{i}}\omega_{\epsilon_n,k}\right| \leq \frac{2}{\epsilon_n}, \;\;\;\; \forall \;\; 1\leq i \leq \lceil n^{2}/4\rceil.
\end{align*}
Let us define the smoothed version of $\mathcal{R}$ as $$\tilde{\mathcal{R}}_z(T_1^k)=\mathcal{R}_{z}(T_1^k)\omega_{\epsilon_n,k},$$
and
$$
\tilde{\Gamma}_{k}^1(z,\bar{\eta}) =4e_k^t \mathbb{E}_k [\tilde{\mathcal{R}}_z(T_1^k)]\mathbb{E}_k [\tilde{\mathcal{R}}_{\bar{\eta}}(T_1^k\conjugate)]e_k.
$$
Since the smoothed version and the original one match everywhere except the narrow annulus $\{\rho_1<\|T_1^k\|<\rho_1+\epsilon_n\}$, we can easily obtain
$$
|\Gamma_{k}^1(z,\bar{\eta})-\tilde{\Gamma}_{k}^1(z,\bar{\eta})| \leq\frac{K}{(\tau - \epsilon_n)^2}1_{\{\rho_1<\|T_1^k\|<\rho_1+\epsilon_n\}},
$$

where $K$ is some universal constant. Now, by Lemma \ref{lemma:Norm bound}, we have, 
\begin{align}
    \mathbb{E}[|\Gamma_{k}^1(z,\bar{\eta})-\tilde\Gamma_{k}^1(z,\bar{\eta})|^2] \leq \frac{K}{(\tau - \epsilon_n)^4}  \exp{\left(-\tau \sqrt{\frac{\alpha n}{2}}\frac{3\rho_1}{4}\right)}.\label{eqn: diff between Gamma and tilde Gamma}
\end{align}

Therefore the smoothing function $\omega$ did not alter the original quantities by significant amount. For the notational convenience, let us denote
\begin{align*}
    \Lambda &:= \mathbb{E}_k[\mathcal{R}_{z}(T_1^k)], & \Upsilon &:= \mathbb{E}_k[\mathcal{R}_{\eta}(T_2^k)^\conjugate],\\
    \tilde{\Lambda} &:= \mathbb{E}_k[\tilde{\mathcal{R}}_z(T_1^k)], &\tilde{\Upsilon} &:= \mathbb{E}_k[\tilde{\mathcal{R}}_\eta(T_2^k)^\conjugate].
\end{align*}

So, $\tilde{\Gamma}_{k}^1(z,\bar{\eta}) = \sum_{s=0}^{\lceil n/2 \rceil} \Tilde{\Lambda}_{ks}\Tilde{\Upsilon}_{sk}.$ Since $x_{ij}s$ satisfy Poincar\'e inequality, we have

\begin{align}
    \Var(\Tilde\Gamma_{k}^1(z,\bar{\eta})) \leq \frac{1}{n\alpha} \sum_{i,j=1}^{k-1}\left[\mathbb{E}\left|\frac{\partial \Tilde\Gamma_{k}^1(z,\bar{\eta})}{\partial t^{1}_{ij}}\right|^2 +\mathbb{E}\left|\frac{\partial \Tilde\Gamma_{k}^1(z,\bar{\eta})}{\partial \bar{t^{1}}_{ij}}\right|^2\right].\label{eqn: variance bound by poincare inequality}
\end{align}
The sum is divided by $n$ because the matrix $T_{1}$ is scaled by $1/\sqrt{n}$. Moreover, the sum stops at $k-1$ because $\Tilde\Gamma_{k}^1(z,\bar{\eta})$ is a constant for $i, j = k,k+1,\dots, \lceil n/2 \rceil$. On the other hand, we have
\begin{align*}
    \frac{\partial \mathcal{R}_{z}(T_1^k)_{ks}}{\partial t_{ij}^1} &= \mathcal{R}_{z}(T_1^k)_{ki} \mathcal{R}_{z}(T_1^k)_{js}\mathbf{1}_{\{j\neq k\}},& 
    \frac{\partial \overline{\mathcal{R}_{z}(T_1^k)_{ks}}}{\partial t_{ij}^1}&=0, 
\end{align*}

As a result,
\begin{align*}
    &\frac{\partial \tilde{\Lambda}_{k s}}{\partial t^{1}_{i j}}=\tilde{\Lambda}_{k i} \tilde{\Lambda}_{j s}\mathbf{1}_{\{j\neq k\}}+\Lambda_{k s} \frac{\partial \omega_{\epsilon_n, k}}{\partial t^{1}_{i j}} 1_{\left\{\rho_1<\|T_{1}^{(k)}\|<\rho_1+\epsilon_n\right\}}, \\
    &\frac{\partial \tilde{\Upsilon}_{s k}}{\partial t^{1}_{i j}}=\Upsilon_{s k} \frac{\partial \omega_{\epsilon_n, k}}{\partial t^{1}_{i j}} 1_{\left\{\rho_1<\|T_{1}^{(k)}\|<\rho_1+\epsilon_n\right\}}, \\
    &\frac{\partial\left(\tilde{\Gamma}_{k}^1(z, \bar{\eta})\right)}{\partial t^{1}_{i j}}=  \sum_{s=0} ^{\lceil n/2 \rceil} \tilde{\Lambda}_{k i} \tilde{\Lambda}_{j s} \tilde{\Upsilon}_{s k}+\sum_{s=0} ^{n/2}\left(\Lambda_{k s } \tilde\Upsilon_{s k}+\tilde{\Lambda}_{k s} \Upsilon_{s k}\right) \frac{\partial \omega_{\epsilon_n, k}}{\partial t^{1}_{i j}} 1_{\left\{\rho_1<\|T_{1}^{(k)}\|<\rho_1+\epsilon_n\right\}}.
\end{align*}

Let us denote $\tilde{y}_{j k}=\sum_{s=0} ^{\lceil n/2 \rceil} \tilde{\Lambda}_{j s} \tilde{\Upsilon}_{s k}$. Then using the facts that $\|\tilde{\Lambda}\|,\|\tilde{\Upsilon}\| \leq\left(\tau-\epsilon_n\right)^{-1}$, we have
\begin{align*}
    &\sum_{i,j=1}^{k-1} \left|\sum_{s=0} ^{\lceil n/2 \rceil} \tilde{\Lambda}_{k i} \tilde{\Lambda}_{j s} \tilde{\Upsilon}_{s k}\right|^2=\sum_{i,j=1}^{k-1} \left|\tilde{\Lambda}_{k i} \tilde{y}_{j k}\right|^2 \leq\left\|\tilde{\Lambda}_k\right\|_2^2\left\|\tilde{y}_k\right\|_2^2 \leq\left(\tau-\epsilon_n\right)^{-6}, \\
    &\sum_{i,j=1}^{k-1} \left|\sum_{s=0} ^{\lceil n/2 \rceil}\left(\Lambda_{ks} \tilde{\Upsilon}_{s k}+\tilde{\Lambda}_{ks} \Upsilon_{s k}\right) \frac{\partial \omega_{\epsilon_n, k}}{\partial a_{i j}} \mathbf{1}_{\left\{\left(\rho_1<\|T_1^k\|<\rho_1+\epsilon_n\right\}\right.}\right|^2 \\
  &\leq \frac{10}{\epsilon_n^2} \sum_{i,j=1}^{k-1} \left\|\tilde{\Lambda}_k\right\|_2^2\left\|\tilde{\Upsilon}_k\right\|_2^2 1_{\left\{\rho<\|T_{1}^{(k)}\|<\rho+\epsilon_n\right\}} \leq 10 n^4\left(\tau-\epsilon_n\right)^{-4} 1_{\left\{\rho_1<\|T_{1}^{(k)}\|<\rho_1+\epsilon_n\right\}} .
\end{align*}

Using the above estimates, \eqref{eqn: event that T_1, T_2 are bounded}, we have
\begin{align}
    &\Var\left(\frac{1}{n} \sum_{k=1}^{\lceil n/2 \rceil} \tilde{\Gamma}_{k}^1(z, \bar{\eta})\right) \notag\\
    &\leq \frac{1}{n} \sum_{k=1}^{\lceil n/2 \rceil} \Var\left(\tilde{\Gamma}_{k}^1(z, \bar{\eta})\right)\notag\\
    &\leq \frac{2}{\alpha n\left(\tau-\epsilon_n\right)^6}+\frac{K n^3}{\left(\tau-\epsilon_n\right)^4} \exp \left(-\sqrt{\frac{\alpha n}{2} }\frac{3 \rho_{1}}{4}\right).\notag
\end{align}

Therefore from \eqref{eqn: diff between Gamma and tilde Gamma}, we obtain
\begin{align}
    &  \Var\left(\frac{1}{n} \sum_{k=1}^{\lceil n/2 \rceil} \Gamma_{k}^1(z, \bar{\eta})\right) \notag\\
    &\leq \frac{1}{n} \sum_{k=1}^{\lceil n/2 \rceil} \mathbb{E}\left[\left|\Gamma_{k}^1(z, \bar{\eta})-\tilde{\Gamma}_{k}^1(z, \bar{\eta})\right|^2\right]+\Var\left(\frac{1}{n} \sum_{k=1}^{\lceil n/2 \rceil} \tilde{\Gamma}_{k}^1(z, \bar{\eta})\right) \notag\\
    &\leq \frac{2}{\alpha n\left(\tau-\epsilon_n\right)^6}+\frac{K n^3}{\left(\tau-\epsilon_n\right)^4} \exp \left(-\sqrt{\frac{\alpha n}{2} }\frac{3 \rho_{1}}{4}\right) \rightarrow 0 \text{ as } n\rightarrow \infty.\notag
\end{align}

Similarly, $ \Var\left(\frac{1}{n} \sum_{k=1}^{\lceil n/2 \rceil} \Gamma_{k}^2(z, \bar{\eta})\right) \rightarrow 0 \text{ as } n \rightarrow \infty.$ Thus,
\begin{align}
    &\Var(\chi_{k}(z, \bar{\eta}))\notag\\
    &\leq   2\Var\left(\frac{1}{n} \sum_{k=1}^{\lceil n/2 \rceil} \Gamma_{k}^1(z, \bar{\eta})\right)+2\Var\left(\frac{1}{n} \sum_{k=1}^{\lceil n/2 \rceil} \Gamma_{k}^2(z, \bar{\eta})\right) 
     \rightarrow 0, \text { as } n\rightarrow \infty .\notag
\end{align}

\end{proof}

This completes the proof of the central limit theorem. We now demonstrate the calculation of the covariance kernel in Section \ref{sec:variance calculation}. Before that let us verify this result through numerical simulation; see Figure \ref{fig:prop_thm1}.
\begin{figure}[H]\label{5}
    \centering
    \includegraphics[scale=0.7]{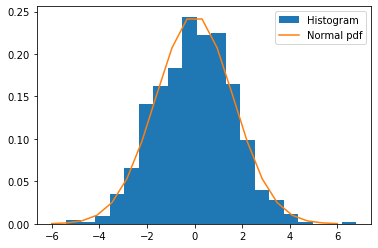}
    \caption{Histogram of $\operatorname{L}^\circ(P_{d})/\sqrt{n}$ of a $4000\times 4000$ random centrosymmetric matrix sampled 1000 times.
    The entries of the matrix are drawn from the Gaussian distribution, where $P_{d}(x)=2x^3+x^4$.}
    \label{fig:prop_thm1}
\end{figure}

%%%%%%%%%%%%%%%%%%%%%%%%%%%%%%%%%%%%%%%%%%%%%%%%%%%%%%%%%%%%%%%%%%%%%%%%%%%
%%%%%%%%%%%%%%%%%%%%%%%%%%%%%%   Variance  %%%%%%%%%%%%%%%%%%%%%%%%%%%%%%%% 
%%%%%%%%%%%%%%%%%%%%%%%%%%%%%%%%%%%%%%%%%%%%%%%%%%%%%%%%%%%%%%%%%%%%%%%%%%%

\section{Calculation of variance}\label{sec:variance calculation}

This section is dedicated to calculating the covariance kernel. We proceed with the expression of the variance in the case of polynomial test functions.
\begin{prop}\label{Prop:Var_complex} 
Let $M$ be the random matrix as defined in Definition \ref{centrodefn} satisfying the Condition \ref{cond:matrixcond}. Then we have
    $$
    \Var(\Tr (P_{d}(M)))  =  \sum_{k=1}^{d} 2ka_{k}^2 + o(1),
     $$
where $P_{d}(x)=\sum_{k=1}^{d} a_{k} x^{k}$ is a polynomial of degree $d \geq 1$ with real coefficients.   
\end{prop}
Notice that, the variance expression does not contain the constant term of the polynomial $P_d(x)$. The constant term of a matrix polynomial will be a constant multiplied by the identity matrix, and this term will not affect the fluctuation result of the LES of a random matrix, as we center the LES by its mean.

We recognize that the essential ingredients are $\mathbb{E}[\Tr(M^{k})\Tr(\bar{M}^{l})]$ and $\mathbb{E}[\Tr(M^{k})]$, where $\bar M$ is obtained by taking complex conjugates of all the entries of $M$. We compute these quantities using combinatorial arguments.
Before proving the Proposition \ref{Prop:Var_complex}, we introduce and analyze some required notions such as pair partitions and chain merging in the following subsections.

\subsection{Analyzing expected product}
We evaluate $\mathbb{E}\left[x_1 x_2 \ldots x_n\right]$, where $x_is$  are random variables from condition \eqref{cond:matrixcond}. Note that $\mathbb{E}[x_i^2] = 0$ if $x_i$ satisfies condition \eqref{cond:matrixcond}.

In the next calculations, we need to evaluate the terms of the forms $\mathbb{E}[\Tr(M^{k})]$ and $\mathbb{E}[\Tr(M^{k})\Tr(\bar{M}^{l})]$. Upon expansion, the term $\Tr(M^{k})$ is represented as $\sum_{i_{1}, i_{2}, \ldots, i_{k}=1}^{n}x_{i_{1}i_{2}}$ $x_{i_{2}i_{3}}$ $\cdots$ $ x_{i_{k}i_{1}}$. Notice that here the random variables $x_{i_{1}i_{2}}$, $x_{i_{2}i_{3}}$, $\dots$ ,$ x_{i_{k}i_{1}}$ are linked via their indices $i_{1}, i_{2}, \ldots, i_{k}$. The arrangement of these indices plays an important role in evaluating expectations. We formalize this using the notation of \textit{index chain} as defined below.

\begin{defn}[Index chains]
The index chain of the product of random variables $x_{i_{1}i_{2}}$ $x_{i_{2}i_{3}}$ $\cdots$ $ x_{i_{k}i_{1}}$ is the unique sequence $i_{1}\to i_{2}\to\cdots\to i_{k}\to i_{1}$ of indices.
    \begin{enumerate}[label=(\alph*)]
        \item Single chain: In the calculation of $\mathbb{E}[\Tr(M^{k})]$, the indices of the product are chained as $i_{1}\to i_{2}\to\cdots\to i_{k}\to i_{1}$ which will be referred as single chain.
        \item Double chain: In the calculation of $\mathbb{E}[\Tr(M^{k})\Tr(\bar{M}^{l})]$, there are two chains of indices $i_{1}\to i_{2}\to \cdots\to i_{k}\to i_{1}$ and $j_{1}\to j_{2}\to \cdots\to j_{l}\to j_{1}$ which will be referred as double chain.
    \end{enumerate}
\end{defn}

In principle, the indices of an index chain can take values from the set $\{1, 2, 3, \ldots, n\}$. However, while calculating $\mathbb{E}[\Tr(M^{k})]$ and $\mathbb{E}[\Tr(M^{k})\Tr(\bar M^{l})]$, we shall see that not all the indices are free to be chosen from $\{1, 2, 3, \ldots, n\}.$ Some indices will be determined by the choices of the other indices. In this context, we shall use the terms `degree of freedom' and `free indices', which are defined in Definition \ref{dof and fv}.

\begin{defn}[Degrees of freedom and free indices]{\label{dof and fv}}
  In a given index chain, free indices are the indices which are not constrained by any other index. The degree of freedom represents the number of free indices in a chain.
  %, such as $i_1, i_2, \dots, i_k$ and $j_1, j_2, \dots, j_l$, from a set $\{1, 2, 3, \dots, n\}$.
\end{defn}
This will be relevant while calculating the $\mathbb{E}[\Tr M^{k}]$ or $\mathbb{E}[\Tr M^{k}\bar{M}^{l}].$ We shall use the notion of free indices in the remaining part of this section; in particular in the proof of Lemma \ref{lemma: one_chain_complex}, Lemma \ref{lemma:equal power_complex}, and Lemma \ref{lemma: cross chain merging}.

Now, we proceed by calculating $\mathbb{E}[\Tr(M^{k})]$ and $\mathbb{E}[\Tr(M^{k})\Tr(\bar{M}^{l})]$ in two different scenarios such as single chain and double chain merging.

\subsection{Single Chain Expectation}\label{sec: single chain expectation}

 In this subsection, we focus on computing the expected value of the trace of a matrix raised to the power of $k$, that is, the expectation of a single chain. Specifically, we are interested in evaluating the expressions of the form.

$$
\mathbb{E}[\mathrm{Tr}(M^{k})] = \frac{1}{n^{ {k/2}}}\sum_{i_{1},i_{2},\dots,i_{k}=1}^{n}\mathbb{E}[x_{i_{1}i_{2}}x_{i_{2}i_{3}} \dots x_{i_{(k-1)}i_{k}}x_{i_{k}i_{1}}].
$$

We state the first lemma in this regard.

\begin{lemma}{\label{lemma: one_chain_complex}}
Let $M$ be the random matrix as defined in Definition \ref{centrodefn} satisfying the Condition \ref{cond:matrixcond}. Then we have the following
\begin{align*}
    \mathbb{E}[\Tr(M^k)] &\leq n^{-k/6}k^{2k}\left(\frac{4}{3\alpha}\right)^{k}\\
    &\leq  k^{3k}n^{-k/6},\;\;\text{if $k>4/3\alpha$}.
\end{align*}
Here $\alpha$ is the constant of Poincar\'e inequality. In particular, if $k\leq n^{\beta}$ for some $\beta < 1/18$, then $\mathbb{E}[\Tr(M^k)]=o(1).$
\end{lemma}
\begin{proof}
This lemma evaluates the limit expectation of a single chain through intra chain merging, which is defined in Section \ref{sec: double chain expectation}. The expectation of the trace of $M^{k}$ can be expressed as
$$
\mathbb{E}[\text{Tr}(M^{k})] = \frac{1}{n^{ k/2 }}\sum_{i_{1},i_{2},\dots,i_{k}=1}^{n}\mathbb{E}[x_{i_{1}i_{2}}x_{i_{2}i_{3}} \dots x_{i_{(k-1)}i_{k}}x_{i_{k}i_{1}}].
$$

To evaluate $\mathbb{E}[x_{i_{1}i_{2}}x_{i_{2}i_{3}} \dots x_{i_{(k-1)}i_{k}}x_{i_{k}i_{1}}]$, we need to determine the number of ways in which $x_{i_{1}i_{2}}x_{i_{2}i_{3}} \dots x_{i_{(k-1)}i_{k}}x_{i_{k}i_{1}}$ simplifies to a product where each $x_{ij}$ may appear with some multiplicity. However, given that $\mathbb{E}[x_{ij}]=0 = \mathbb{E}[x_{ij}^{2}]$ according to Condition \ref{cond:matrixcond}, we should focus on the cases where more than two random variables are equal. Since the product has $k$ random variables, essentially the the task boils down to considering the integer partitions of $k$, where group sizes in each partition is at least $3$. For example, $22=10+5+4+3$ is an integer partition of $22$ having four groups, where group sizes are $10,5,4,3.$

Suppose in an integer partition of $k$, we have $g$ groups, where each group contains $k_1, k_2, \dots, k_g$ random variables respectively, and the random variables within any particular group are equal. For example, a group of size $k_i$ consists of $k_i$ equal random variables, and no two groups share the same random variables. Additionally, the cardinality of each group is at least 3, and $\sum_{l=1}^{g} k_l = k$. 

Let us analyze a generic group of size $k_{l}.$ To reduce notional complexity, here we denote the group as $x_{\alpha_{1}}x_{\alpha_{2}}\ldots x_{\alpha_{k_{l}}}$, where each $\alpha_{p}$ is of the form $(i_{\alpha_{p}}, i_{\alpha_{p}+1}).$ As per the construction, $x_{\alpha_{1}} = x_{\alpha_{2}} = \cdots = x_{\alpha_{k_{l}}}.$ Therefore using the fact that $x_{ij}$s satisfy Poincar\'e inequality, we have $$\mathbb{E}[x_{\alpha_{1}}x_{\alpha_{2}}\ldots x_{\alpha_{k_{l}}}]= \mathbb{E}[|x_{11}|^{k_{l}}]\leq k_{l}!\left(\frac{2}{\alpha}\right)^{k_{l}/2}.$$

Hence we conclude that if we fix a partition $\sum_{l=1}^{g}k_{l} = k$, and fix random variables in each group of the partition, then
\begin{align*}
    \mathbb{E}[x_{i_{1}i_{2}}x_{i_{2}i_{3}} \dots x_{i_{(k-1)}i_{k}}x_{i_{k}i_{1}}] &\leq  \prod_{l=1}^{g}k_{l}!\left(\frac{2}{\alpha}\right)^{k_{l}/2}\\
    &=\left(\frac{2}{\alpha}\right)^{k}\prod_{l=1}^{g}k_{l}!\\
    &\leq \left(\frac{2}{\alpha}\right)^{k}k!,
\end{align*}
where the last inequality follows from the fact that $m!n!\leq (m+n)!.$
However, if we fix a partition of $k$, but do not fix the random variables in the partition groups, then we have some choices for the indices $\alpha_{1}, \alpha_{2}, \ldots, \alpha_{k_{l}}$ in the product $x_{\alpha_{1}}x_{\alpha_{2}}\ldots x_{\alpha_{k_{l}}}$, which is a group of size $k_{l}$. Now since all the random variables in this group are equal, we only have the freedom to choose $\alpha_{1}.$ Once $\alpha_{1}$ is fixed, the remaining $\alpha_{p}$s in the same group can either be $\alpha_{1}$ or $n+1 - \alpha_{1}$. Hence for a given $\alpha_{1}$, we have $2^{k_{l}-1}$ many possibilities of the indices $\alpha_{2}, \alpha_{3}, \ldots, \alpha_{k_{l}}$ for that group.

Now the index $\alpha_{1}=(i_{\alpha_{1}}, i_{\alpha_{1}+1})$ consists of two indices. However, since this group $x_{\alpha_{1}}x_{\alpha_{2}}\ldots x_{\alpha_{k_{l}}}$ is a part of the full chain $x_{i_{1}i_{2}}x_{i_{2}i_{3}} \dots x_{i_{(k-1)}i_{k}}x_{i_{k}i_{1}}$, one of the indices among $i_{\alpha_{1}}, i_{\alpha_{1}+1}$ must also appear in some other group. As a result, either $i_{\alpha_{1}}$ or $i_{\alpha_{1}+1}$ must be determined by some other group. Therefore we have the freedom to choose only one of the indices $i_{\alpha_{1}}, i_{\alpha_{1}+1}$, which can be done in at most $n$ possible ways. Hence we have an additional factor $n2^{k_{l}-1}$ in the $l$th group of the partition. And overall, we have a factor of $n^{g}2^{\sum_{l=1}^{g}(k_{l}-1)}=n^{g}2^{k-g}.$

Finally, combining all the above arguments we have
\begin{align*}
    \mathbb{E}[\text{Tr}(M^{k})] &= \frac{1}{n^{ k/2 }}\sum_{i_{1},i_{2},\dots,i_{k}=1}^{n}\mathbb{E}[x_{i_{1}i_{2}}x_{i_{2}i_{3}} \dots x_{i_{(k-1)}i_{k}}x_{i_{k}i_{1}}]\\
    &=\frac{1}{n^{k/2}}\sum_{g=1}^{k/3}\sum_{\stackrel{k_{1}+k_{2}+\cdots+k_{g}=k}{k_{l}\geq 3}}n^g 2^{k-g}\left(\frac{2}{\alpha}\right)^{k}k!\\
    &\leq \frac{1}{n^{k/2}}\left(\frac{k}{3}\right)^{k}n^{k/3}2^{2k/3}\left(\frac{2}{\alpha}\right)^{k}k!\\
    &\leq n^{-k/6}k^{2k}\left(\frac{4}{3\alpha}\right)^{k}\\
    &\leq n^{-k/6}k^{3k},
\end{align*}
for $k>4/3\alpha$. The third last inequality follows from the fact that $n^g 2^{k-g}\leq n^{k/3}2^{2k/3}$ for $1\leq g\leq k/3$ and the number of partitions of $k$ into at most $k/3$ many groups is $(k/3)^{k}.$ This completes the proof.
\end{proof}

We shall now proceed with the calculation of the covariance between $\Tr M^{k}$, $\Tr M^{l}.$

%%%%%%%%%% Double Chain %%%%%%%%%%%%%%%%%%%%%%

\subsection{Double chain Expectation}\label{sec: double chain expectation}
In this section, our focus is on computing the expected value of the product of traces of matrices raised to different powers, which requires the computation of the expectation of a double chain. Specifically, we are interested in expressions of the form

$$
\sum_{\substack{ k,l=1 }}^d \frac{1}{n^{(k+l)/2}} \sum_{{i_{1},i_{2},\dots i_{k} , j_{1}, j_{2}, \dots j_{l}=1}}^{n} \mathbb{E}\left[(x_{i_{1}i_{2}}x_{i_{2}i_{3}}\dots x_{i_{k}i_{1}})(\bar x_{j_{1}j_{2}}\bar x_{j_{2}j_{3}}\dots \bar x_{j_{l}j_{1}})\right].
$$
Here, $k$ and $l$ are positive integers denoting the lengths of the respective chains. From the condition \ref{cond:matrixcond}, we know that $\mathbb{E}[x_{ij}]=0=\mathbb{E}[x_{ij}^{2}].$ Therefore as explained in Lemma \ref{lemma: one_chain_complex}, we need to choose the indices of the chains in such such a way that each random variable occurs with a multiplicity. However, it is quite evident that if we make more random variables equal to each other within a chain or across different chains, then we lose more degrees of freedom in the index chains. Therefore, we aim to make fewer random variables equal, preferably $2$, across two different chains. Since in this case, we have the term $\mathbb{E}[|x_{ij}|^2]$, which is $1$. Thus, when we have a product of two chains such as $(x_{i_{1}i_{2}}x_{i_{2}i_{3}}\dots x_{i_{k}i_{1}})(\bar x_{j_{1}j_{2}}\bar x_{j_{2}j_{3}}\dots \bar x_{j_{l}j_{1}})$, we can pair random variables from the first chain $x_{i_{1}i_{2}}x_{i_{2}i_{3}}\dots x_{i_{k}i_{1}}$ with the second chain $\bar x_{j_{1}j_{2}}\bar x_{j_{2}j_{3}}\dots \bar x_{j_{l}j_{1}}$ in the following three different ways:
\begin{enumerate}
    \item All the random variables of one chain find their pair within the chain itself - \textit{intra chain merging}.
    \item All random variables in one chain find their pair in the other chain - \textit{cross chain merging}.
    \item Some of the random variables of one chain find their pairs in the chain itself, and others find in the other chain - \textit{partial chain merging}.
\end{enumerate}
Let us explain each case as follows.

\textbf{Case 1 (Intra chain merging):} As we have seen in the Lemma \ref{lemma: one_chain_complex}, the single chain expectation is asymptotically zero. Therefore, intra chain merging gives us zero asymptotic contribution.

%%%%%%%%%%%%%%%%%% Cross Chain Merging %%%%%%%%%%%%%%%

\textbf{Case 2 (Cross chain merging):}\label{case 3: cross chain merging}
%However, notice that the above three possibles cases depend on the lengths of the chains. This is elaborated further in the Lemmas \ref{lemma:two_chain_lenth_neq} and \ref{lemma:two_chain_lenth_eq}.
Let us consider a simple example to evaluate $\mathbb{E}[\Tr(M^{k})\Tr(\bar{M}^{l})]$ with $k=l=5$.
Thus,
$$
\mathbb{E}[\Tr(M^{5})\Tr(\bar{M}^{5})] = \frac{1}{n^{5}}\sum_{\substack{i_{1},i_{2},\dots i_{k},\\  j_{1}, j_{2}, \dots j_{5}=1}}^{n} \mathbb{E}\left[(x_{i_{1}i_{2}}x_{i_{2}i_{3}}x_{i_{3}i_{4}} x_{i_{4}i_{5}}x_{i_{5}i_{1}})(\bar{x}_{j_{1}j_{2}}\bar{x}_{j_{2}j_{3}}\bar{x}_{j_{3}j_{4}}\bar{x}_{j_{4}j_{5}} \bar{x}_{j_{5}j_{1}})\right].
$$
As we shall see in the proof of Lemma \ref{lemma:equal power_complex}, we get maximum contribution if every random variable from the first chain gets paired with a random variable from the second chain.

One way of cross chain merging is illustrated in Figure \ref{fig:k=5,l=5,Crossing-case_complex}, and we demonstrate that it gives us a nonzero asymptotic contribution.

\begin{figure}[h]
    \begin{center}
\begin{tikzpicture}

   \def\x{0}
   \def\y{0}
  \filldraw (\x, \y) circle (0.06);
   \node[anchor = south] at (\x, \y) {$x_{i_1i_2}$};
  \filldraw (\x+1, \y) circle (0.06);
   \node[anchor = south] at (\x+1, \y) {$x_{i_2i_3}$};
  \filldraw (\x+2, \y) circle (0.06);
   \node[anchor = south] at (\x+2, \y) {$x_{i_3i_4}$};
  \filldraw (\x+3, \y) circle (0.06);
   \node[anchor = south] at (\x+3, \y) {$x_{i_4i_5}$};
  \filldraw (\x+4, \y) circle (0.06);
   \node[anchor = south] at (\x+4, \y) {$x_{i_5i_1}$};
  \filldraw (\x+5, \y) circle (0.06);
   \node[anchor = south] at (\x+5, \y) {$\bar{x}_{j_1j_2}$};
  \filldraw (\x+6, \y) circle (0.06);
   \node[anchor = south] at (\x+6, \y) {$\bar{x}_{j_2j_3}$};
  \filldraw (\x+7, \y) circle (0.06);
   \node[anchor = south] at (\x+7, \y) {$\bar{x}_{j_3j_4}$};
  \filldraw (\x+8, \y) circle (0.06);
  \node[anchor = south] at (\x+8, \y) {$\bar{x}_{j_4j_5}$};
   \filldraw (\x+9, \y) circle (0.06);
  \node[anchor = south] at (\x+9, \y) {$\bar{x}_{j_5j_1}$};
 \draw[thick] (\x, \y) -- (\x, \y-0.5) -- (5+\x, \y-0.5) -- (\x+5, \y); 
 \draw[thick] (\x+1, \y) -- (\x+1, \y-0.7) -- (6+\x, \y-0.7) -- (\x+6, \y);
 \draw[thick] (\x+2, \y) -- (\x+2, \y-0.9) -- (7+\x, \y-0.9) -- (\x+7, \y);
  \draw[thick] (\x+3, \y) -- (\x+3, \y-1.1) -- (8+\x, \y-1.1) -- (\x+8, \y);
  \draw[thick] (\x+4, \y) -- (\x+4, \y-1.3) -- (9+\x, \y-1.3) -- (\x+9, \y);
  \end{tikzpicture}
  \end{center}
  \caption{Sequential cross chain merging}
    \label{fig:k=5,l=5,Crossing-case_complex}
\end{figure}
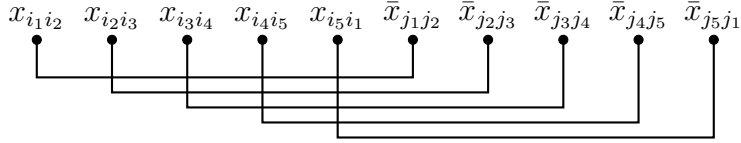

For any pair of random variables $(x_{ij}, \bar{x}_{kl})$, where $x_{ij}$ and $\bar{x}_{kl}$ are from a chain $x_{i_{1}i_{2}}x_{i_{2}i_{3}}x_{i_{3}i_{4}}$
$ x_{i_{4}i_{5}}x_{i_{5}i_{1}}$, and $\bar{x}_{j_{1}j_{2}}\bar{x}_{j_{2}j_{3}}\bar{x}_{j_{3}j_{4}}\bar{x}_{j_{4}j_{5}} \bar{x}_{j_{5}j_{1}}$ respectively, $\mathbb{E}[x_{ij}\bar{x}_{kl}]\neq 0$ if and only if $(i, j) = (k, l)$ or $(i, j) = (n+1 - k, n+1 - l)$. This leads to the following possible pairings in the above example.
  
$$
\begin{array}{l}
\left\{\begin{array} { l } 
{ j _ { 1 } = n + 1 - i _ { 1 } } \\
{ j _ { 2 } = n + 1 - i _ { 2 } }
\end{array} \quad \text { or } \quad \left\{\begin{array}{l}
j_1=i_1 \\
j_2=i_2
\end{array}\right.\right. \\\\

\left\{\begin{array} { l } 
{ j _ { 2 } = n + 1 - i _ { 2 } } \\
{ j _ { 3 } = n + 1 - i _ { 3 } }
\end{array} \quad \text { or } \quad \left\{\begin{array}{l}
j_2=i_2 \\
j_3=i_3
\end{array}\right.\right. \\\\

\left\{\begin{array} { l } 
{ j _ { 3 } = n + 1 - i _ { 3 } } \\
{ j _ { 4 } = n + 1 - i _ { 4 } }
\end{array} \quad \text { or } \quad \left\{\begin{array}{l}
j_3=i_3 \\
j_4=i_4
\end{array}\right.\right. \\\\

\left\{\begin{array} { l } 
{ j _ { 4 } = n + 1 - i _ { 4 } } \\
{ j _ { 5 } = n + 1 - i _ { 5 } }
\end{array} \quad \text { or } \quad \left\{\begin{array}{l}
j_4=i_4 \\
j_{5}=i_5
\end{array}\right.\right. \\\\

\left\{\begin{array} { l } 
{ j _ { 5 } = n + 1 - i _ { 5 } } \\
{ j _ { 1 } = n + 1 - i _ { 1 } }
\end{array} \quad \text { or } \quad \left\{\begin{array}{l}
j_{5}=i_5 \\
j_{1}=i_1
\end{array}\right.\right. \\
\end{array}
$$

If we start with $j_1 = n+1 - i_{1}$, then the chain has to follow the constraints listed in the left column; not the ones which are listed on the right column. Similarly, if we start with the condition $j_{1} = i_{1}$, then we have to follow the constraints that are listed on the right column. Mixing the conditions listed in the left and right column will reduce the number of free indices, leading to an asymptotically zero contribution, as the full term is divided by $n^{5}$. 

For example, if we take $(i_{1}, i_{2}) = (n+1 - j_{1}, n+1 - j_{2})$ and decide to take $(i_{2}, i_{3}) = (j_{2}, j_{3})$, then to keep the continuity of the chain, we must have $j_{2} = i_{2} = n+1 - j_{2}$ which enforces one extra constraint $j_{2} = n+1 - j_{2}$. Thus, we have lost a degree of freedom there. Now, if the indices follow the constraints listed on the first column, then there are $n^{5}$ many choices for choosing all the indices. For each such choice, the chain will reduce to a product of the terms of the form $\mathbb{E}[|x_{ij}|^{2}] (=1).$

Thus, we conclude that when we do the cross chain merging for the first and the second chain, namely $i_{1}\rightarrow i_{2}\rightarrow i_{3}\rightarrow i_{4}\rightarrow  i_{5}\to i_{1}$ and $j_{1}\rightarrow j_{2}\rightarrow j_{3}\rightarrow j_{4} \rightarrow j_{5}\to j_{1}$, one way for cross chain merging is to copy the first chain in place of the second chain, resulting in: $i_{1}\rightarrow i_{2}\rightarrow i_{3}\rightarrow i_{4}\rightarrow  i_{5}\to i_{1}\rightarrow i_{2}\rightarrow i_{3}\rightarrow i_{4}\rightarrow  i_{5}\to i_{1}$. This represents one possibility with $5$ free indices. Additionally, we can start the second chain with $i_{2}$ or $i_{3}$ or $\dots$ or $i_{5}$ like $i_{2}\rightarrow i_{3}\rightarrow \dots \rightarrow i_{5} \rightarrow i_{1}\rightarrow i_{2}$, resulting in $4$ more choices. Therefore, we have $5 n^{5}$ combinations for cross chain merging of the first and the second chain so far. Moreover, since $x_{(i,j)}=x_{(n+1-i,n+1-j)}$ , we can perform a cross chain merging by changing the second chain to $(n+1-i_{1})\rightarrow (n+1-i_{2})\rightarrow \dots \rightarrow (n+1-i_{5})\rightarrow (n+1-i_{1})$, which will contribute another $5n^{5}.$
%representing another one possibility with $5$ free variables \textcolor{blue}{indices}. Additionally, we have $4$ more choices by rotating the second sequence, similar to the previous case.
Therefore, the total number of combinations for cross chain merging of the first and second chain is $2 \times 5n^{5}$ from type of merging in Figure \ref{fig:k=5,l=5,Crossing-case_complex}.

So far, we have observed that if $x_{i_ti_l}$ finds its pair $\bar{x}_{j_pj_q}$, then the pairing for $x_{i_{t+1}i_{l+1}}$ is $\bar{x}_{j_{p+1}j_{q+1}}$ for all $t$ and $l$ such that $1 \leq t, l \leq 5$. In other words, the random variables from the first chain match with the random variables in the second chain in the same order as they appear in the first chain. Now, let us consider the another type of case as demonstrated in Figure \ref{fig:k=5,l=5,other_Crossing-case_complex}, where the first chain is not merging with the second chain in the same sequence. We shall argue that such mergings will give negligible contributions.

\begin{figure}[h]
    \begin{center}
\begin{tikzpicture}

   \def\x{0}
   \def\y{0}
  \filldraw (\x, \y) circle (0.06);
   \node[anchor = south] at (\x, \y) {$x_{i_1i_2}$};
  \filldraw (\x+1, \y) circle (0.06);
   \node[anchor = south] at (\x+1, \y) {$x_{i_2i_3}$};
  \filldraw (\x+2, \y) circle (0.06);
   \node[anchor = south] at (\x+2, \y) {$x_{i_3i_4}$};
  \filldraw (\x+3, \y) circle (0.06);
   \node[anchor = south] at (\x+3, \y) {$x_{i_4i_5}$};
  \filldraw (\x+4, \y) circle (0.06);
   \node[anchor = south] at (\x+4, \y) {$x_{i_5i_1}$};
  \filldraw (\x+5, \y) circle (0.06);
   \node[anchor = south] at (\x+5, \y) {$\bar{x}_{j_1j_2}$};
  \filldraw (\x+6, \y) circle (0.06);
   \node[anchor = south] at (\x+6, \y) {$\bar{x}_{j_2j_3}$};
  \filldraw (\x+7, \y) circle (0.06);
   \node[anchor = south] at (\x+7, \y) {$\bar{x}_{j_3j_4}$};
  \filldraw (\x+8, \y) circle (0.06);
  \node[anchor = south] at (\x+8, \y) {$\bar{x}_{j_4j_5}$};
   \filldraw (\x+9, \y) circle (0.06);
  \node[anchor = south] at (\x+9, \y) {$\bar{x}_{j_5j_1}$};
 \draw[thick] (\x, \y) -- (\x, \y-0.5) -- (7+\x, \y-0.5) -- (\x+7, \y); 
 \draw[thick] (\x+1, \y) -- (\x+1, \y-0.7) -- (9+\x, \y-0.7) -- (\x+9, \y);
 \draw[thick] (\x+2, \y) -- (\x+2, \y-0.9) -- (5+\x, \y-0.9) -- (\x+5, \y);
  \draw[thick] (\x+3, \y) -- (\x+3, \y-1.1) -- (8+\x, \y-1.1) -- (\x+8, \y);
  \draw[thick] (\x+4, \y) -- (\x+4, \y-1.3) -- (6+\x, \y-1.3) -- (\x+6, \y);
  \end{tikzpicture}
  \end{center}
  \caption{Non sequential cross chain merging}
    \label{fig:k=5,l=5,other_Crossing-case_complex}
\end{figure}
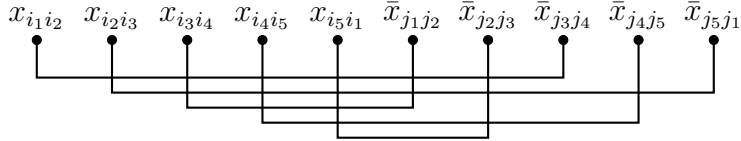
For this case, we need to impose more constraints that will reduce the degree of freedom and leads to zero asymptotic contribution. Here, 
$x_{i_{1}i_{2}}=\bar{x}_{j_{3}j_{4}}$ if and only if $(i_1,i_{2}) = (j_{3}, j_4)$ or the dual case $(i_2,i_{3}) = ((n+1)-j_{3}, (n+1)-j_4)$.

If $j_{3}=i_1$, this implies $j_{4} = i_2$. Furthermore, since we are pairing $x_{i_{2}i_{3}}$ and $\bar{x}_{j_{5}j_{1}}$, we shall have $j_{5} = i_{2}$ and $i_3=j_1$, or the dual case. Now, look at the consecutive random variables of those already paired up. We observe that to maintain the continuity of the chain, we must impose restrictions that will significantly reduce the number of free variables indices and lead to zero asymptotic contribution.

Therefore, the total number of combinations that provide a nonzero asymptotic contribution is $2 \times 5n^{5}$. We obtain that contribution from the cases as demonstrated in the Figure \ref{fig:k=5,l=5,Crossing-case_complex}.
Consequently,
$$
\lim_{{n \to \infty}} \mathbb{E}[\Tr(M^{5})\Tr(\bar{M}^{5})]
 =\lim_{{n \to \infty}} \frac{2\times 5 n^5+o(n^{5})}{n^{5}}= 2\times 5.
$$

Now, we generalize the above result in the following lemma.
\begin{lemma}{\label{lemma:equal power_complex}}
Let $M$ the random matrix as defined in Definition \ref{centrodefn} satisfying the Condition \ref{cond:matrixcond}. Then we have 
    \begin{align}
      \mathbb{E}[\Tr(M^{k})\Tr(\bar{M}^{k})]
       &=2k+o(1),\notag
    \end{align}
    if $k<n^{\gamma}$ for $\gamma<1/20.$
    \end{lemma}  
\begin{proof}The lemma estimates the limiting expectation of  cross chain merging. Consider
\begin{align*}
\mathbb{E}[\Tr(M^{k})\Tr(\bar{M}^{k})] &= \frac{1}{n^{k}}\mathbb{E}\left[\sum_{{i_{1},i_{2},\dots i_{k} \atop j_{1}, j_{2}, \dots j_{k} = 1}}^{n}(x_{i_{1}i_{2}}x_{i_{2}i_{3}}\dots x_{i_{k}i_{1}})(\bar{x}_{j_{1}j_{2}}\bar{x}_{j_{2}j_{3}}\dots \bar{x}_{j_{k}j_{1}})\right]\\
&= \frac{1}{n^{k}} \sum_{{i_{1},i_{2},\dots i_{k} \atop j_{1}, j_{2}, \dots j_{k}= 1}}^{n}\mathbb{E}\left[(x_{i_{1}i_{2}}x_{i_{2}i_{3}}\dots x_{i_{k}i_{1}})(\bar{x}_{j_{1}j_{2}}\bar{x}_{j_{2}j_{3}}\dots \bar{x}_{j_{k}j_{1}})\right].
\end{align*}

The main contribution will come from two possible cases, such as (a) the chain  $i_{1}\rightarrow i_{2}\rightarrow \dots \rightarrow i_{k}\to i_{1}$ overlaps completely with the chain $j_{1}\rightarrow j_{2}\rightarrow \dots \rightarrow j_{k}\to j_{1}$, or (b) it overlaps with the dual version of the second chain, which is $(n+1-j_{1})\rightarrow (n+1-j_{2})\rightarrow \dots \rightarrow (n+1-j_{k})\rightarrow (n+1-j_{1}).$ In either cases, we have $k$ many possible ways to start the second chain. Once the chains overlap, it reduces to the product of terms of the form $\mathbb{E}[|x_{ij}|^{2}]$ $(=1)$. Such complete overlapping will contribute $2kn^{k}.$ 

Now we argue that if the chains do not overlap completely, then the contribution is negligible. Let us assume that only $p$ random variables from the the first chain overlap with $p$ random variables in the second chain. In that case, the overlapping part contributes $2pn^{p}$. Additionally, the extra $k-p$ random variables from the first and the second chain may contribute at most $n^{(k-p-1)/3}(k-p)^{3(k-p)}$ each. This a consequence of the result stated in the Lemma \ref{lemma: one_chain_complex}. Note that in the Lemma \ref{lemma: one_chain_complex}, we had the factor $n^{-k/6}$, where $k$ was the length of the chain, and the trace sum was already divided by $n^{k/2}$. Here length of the chain is $k-p$, but it is not divided by $n^{(k-p)/2}$ yet. Consequently, we are supposed to get $n^{(k-p)/3}$ instead of $n^{-(k-p)/6}.$ However, to keep the continuity of the full chain (e.g. $x_{i_{1}i_{2}}x_{i_{2}i_{3}}\dots x_{i_{k}i_{1}}$), one of the index among $k-p$ variables is already determined by the remaining $p$ variables which are matched with $p$ variables of the another chain (e.g. $x_{j_{1}j_{2}}x_{j_{2}j_{3}}\dots x_{j_{k}j_{1}}$). As a result, we have the factor $n^{(k-p-1)/3}.$ Hence in such a case, the combined contribution turns out to be at most 
\begin{align}\label{eqn: negligible terms in cross chain merging (part 1)}
    2pn^{p}n^{2(k-p-1)/3}(k-p)^{6(k-p)}&=n^{k} \times 2pn^{-(k-p-2)/3}(k-p)^{6(k-p)}.
    %&=o(n^{k}),\;\;\;\text{if $p<k$.}
\end{align}
Additionally, the starting points of the $p$-chains can be set to one of indices among $i_{1}, i_{2}, \ldots, i_{k}$ (for the first chain) or $j_{1}, j_{2}, \ldots, j_{k}$ (for the second chain). Therefore multiplying by a factor of $k^{2}$ and summing up over all possible values of $p$, we see that if the full chains do not overlap, then we have the following contribution
\begin{align}\label{eqn: negligible terms in cross chain merging}
    n^{k}\sum_{p=1}^{k-3}2pk^{2}n^{-(k-p-2)/3}(k-p)^{6(k-p)}.
\end{align}
Note that $p$ can not be $k-1$ or $k-2$. Because in that case, the remaining chain would be of length one or two, and from the Condition \ref{cond:matrixcond}, we have $\mathbb{E}[x_{ij}] = 0 = \mathbb{E}[x_{ij}^{2}].$ The above sum is $o(n^{k})$ if $k<n^{\gamma}$ for $\gamma< 1/20.$
Therefore, we have

\begin{align*}
    &\mathbb{E}[\Tr(M^{k})\Tr(\bar{M}^{k})]\\
    &=\frac{1}{n^k}\sum_{\substack{i_{1},i_{2},\dots i_{k},\\  j_{1}, j_{2}, \dots j_{k}=1}}^{n} \mathbb{E}\left[(x_{i_{1}i_{2}}x_{i_{2}i_{3}}\dots x_{i_{k}i_{1}})(\bar{x}_{j_{1}j_{2}}\bar{x}_{j_{2}j_{3}}\dots \bar{x}_{j_{k}j_{1}})\right]\\
    &=\frac{2kn^{k}}{n^{k}}+o(1)\\
    &=2k+o(1).
\end{align*}
\end{proof}

%%%%%%%%%%%%%%%%%% partial Chain Merging %%%%%%%%%%%%%%%

\textbf{Case 3 (Partial chain merging):}\label{case 2: partial chain merging} If the chain lengths are not equal, then we have a partial chain merging, which is explained in the following lemma.

\begin{lemma}\label{lemma: cross chain merging}
    Let $M$ be the random matrix as defined in Definition \ref{centrodefn}, and satisfying the Condition \ref{cond:matrixcond}. Then for $k\neq l$, we have
    \begin{align*}
        \mathbb{E}[\Tr(M^{k})\Tr(\bar{M}^{l})] \leq 2kln^{-|l-k|/6}|l-k|^{3|l-k|} + o(1),
    \end{align*}
    if $k,l<n^{\gamma}$ for $\gamma<1/20.$
\end{lemma}

\begin{proof}
    Without loss of generality, let us assume that $k<l$. We adopt the methodology from Lemma \ref{lemma:equal power_complex}. Let a subchain of length $p$ from the first chain $x_{i_{1}i_{2}}x_{i_{2}i_{3}}\dots x_{i_{k}i_{1}}$ gets mathed to a subchain of length $p$ from the second chain$ x_{j_{1}j_{2}}x_{j_{2}j_{3}}\dots x_{j_{l}j_{1}}.$ In that case, following the equations \eqref{eqn: negligible terms in cross chain merging (part 1)} and \eqref{eqn: negligible terms in cross chain merging}, we have
    \begin{align*}
        \mathbb{E}[\Tr M^{k}\bar{M}^{l}]&\leq\frac{1}{n^{(k+l)/2}}\sum_{p=1}^{k-3}2pn^{p}\times kn^{(k-p-1)/3}(k-p)^{3(k-p)}\times l n^{(l-p-1)/3}(l-p)^{3(l-p)}\\
        &+\frac{1}{n^{(k+l)/2}}2kn^{k}\times ln^{(l-k-1)/3}(l-k)^{3(l-k)}\\
        &\leq n^{-(l+k)/2}\sum_{p=1}^{k}2pkln^{(k+l+p-2)/3}(k-p)^{3(k-p)}(l-p)^{3(l-p)}\\
        &+2kln^{-(l-k)/6}(l-k)^{3(l-k)}\\
        &\leq 2kln^{-(l-k)/6}(l-k)^{3(l-k)} + o(1).
    \end{align*}
    The $o(1)$ in the last inequality is a consequence of $k,l\leq n^{\gamma}, \gamma<1/20.$
\end{proof}

\begin{proof}
Here the chain lengths are not equal. Without loss of generality, let us assume that $k < l$. 
We can divide the second chain into two parts - one with the length of the first chain and one with the remaining elements, as follows;
$(x_{i_{1}i_{2}}x_{i_{2}i_{3}}\dots x_{i_{k}i_{1}})$ $(\bar{x}_{j_{1}j_{2}}\bar{x}_{j_{2}j_{3}}\dots \bar{x}_{j_{k}j_{k+1}})$ $(\bar{x}_{j_{k+1}j_{k+2}} \bar{x}_{j_{k+2}j_{k+3}}\dots \bar{x}_{j_{l}j_{1}})$. To get the maximum contribution, the first and second chains will have cross chain merging, and the third chain will have intra chain merging.

From the cross chain merging of the first and second chains, we have $2kn^k$ combinations; see the Lemma \ref{lemma:equal power_complex}. Now, we have one option for the third chain, which is intra chain merging. We have at most $n^{(l-k-1)/3}(l-k)^{3(l-k)}$ combinations from intra chain merging. Note that the random variable $j_{k+1}$ has already been fixed in the previous case where we merged the first and second chains. Additionally, we chose the second and third chains from a single chain, and this can be done in $l$ many ways. Therefore, we have at most $ln^{(l-k-1)/3}(l-k)^{3(l-k)}$ combinations. Thus, the total number of combinations is of order $O(lk(l-k)^{3(l-k)}n^{(l+2k-1)/3})$. However, when evaluating $\mathbb{E}[\Tr(M^{k})\Tr(\bar{M}^{l})]$, the expression $\mathbb{E}\left[(x_{i_{1}i_{2}}x_{i_{2}i_{3}}\dots x_{i_{k}i_{1}})(\bar x_{j_{1}j_{2}}\bar x_{j_{2}j_{3}}\dots \bar x_{j_{l}j_{1}})\right]$ is multiplied by $(1/n)^{(k+l)/2}$. Therefore, from this case we have a contribution of $O((l-k)^{3(l-k)}n^{-(l-k)/6})$.
\end{proof}

We summarise all the above cases as follows.

\begin{rem}{\label{rem: intra and partial chain merging}}
    While calculating $\mathbb{E}[\Tr(M^{k})\Tr(\bar{M}^{l})]$, the significant contribution comes only when $k = l$, see Lemma \ref{lemma:equal power_complex}. From Lemma \ref{lemma: one_chain_complex} and Lemma \ref{lemma: cross chain merging}, we conclude that if $k \neq l$, although intra-chain merging and partial-chain merging are possible strategies, asymptotically they provide zero contributions. In particular, $\mathbb{E}[\Tr(M^{k})\Tr(\bar{M}^{l})] = O((l-k)^{3(l-k)}n^{-(l-k)/6})$.

\end{rem}

Now, we are ready to summarize the above results into the proof of Proposition \ref{Prop:Var_complex}.
\begin{proof}[Proof of Proposition \ref{Prop:Var_complex}]
We may write the variance as 
\begin{align} \label{eq:Var_1_complex}
&\Var(\Tr(P_{d}(M))) \notag \\
&= \Var\left(\sum_{k=1}^{d}a_{k} \Tr(M^{k})\right) \notag \\
&= \sum_{k,l=1}^{d}  a_{k} a_{l} \Cov\left(\Tr(M^k),\Tr(M^l)\right) \notag \\
&= \sum_{k,l=1}^{d} a_{k} a_{l}\left(\mathbb{E}\left[\Tr(M^k)\Tr (\bar{M}^l)\right]-\mathbb{E}\left[\Tr(M^k)\right] \mathbb{E}\left[\Tr(\bar{M}^l)\right]\right) \notag\\
&=\sum_{\substack{k\neq l \\ k,l=1}}^d 0 
+ \sum_{\substack{k = l \\ k,l=1}}^d2ka_{k}^2+o(1)\\
&= \sum_{k=1}^d2k a_{k}^2+o(1).\notag
\end{align}

The equality in \eqref{eq:Var_1_complex} follows from the Lemmas \ref{lemma: one_chain_complex}, \ref{lemma:equal power_complex}, and Remark \ref{rem: intra and partial chain merging}.
\end{proof}

\subsection{Covariance kernel of $\Tr \hat{\mathcal{R}}_{z}(M)$}\label{sec: calculation of the covariance kernel} In this subsection, we find the covariance kernel of the Gaussian process as mentioned in Proposition \ref{prop: Gaussian process and tightness}. As defined at the beginning of Section \ref{sec: proof of the clt}, let $\rho=\max\{\rho_{1}, \rho_{2}\}$. Now, on the event $\Omega_{n}$, we can expand $\Tr \hat{\mathcal{R}}_{z}(M)$ on the boundary $\partial \mathbb{D}_{\rho+\tau}$ as follows
\begin{align*}
    \Tr \hat{\mathcal{R}}_{z}(M)=\frac{n}{z}+\sum_{k=1}^{\lfloor \log^{2} n\rfloor - 1}z^{-k-1}\Tr M^{k} + \frac{1}{z^{\lfloor \log^{2} n\rfloor +1}}\Tr[\hat{\mathcal{R}}_{z}(M)M^{\lfloor \log^{2} n \rfloor}].
\end{align*}

We see that on the event $\Omega_{n}$, the last term is bounded by $\frac{n}{\tau}(\rho/|z|)^{\lfloor \log^{2} n \rfloor} = o(1).$ Now using Remark \ref{rem: intra and partial chain merging}, we have

\begin{align*}
    \mathbb{E}[\Tr M^{k}\Tr \bar M^{l}]=O(|k-l|^{3|k-l|}n^{-|k-l|/6}).
\end{align*}
Furthermore, since $(x^3n^{-1/6})^{x}$ is a decreasing function of $x$ for $1\leq x \leq n^{1/18}/e$, we can estimate
\begin{align*}
    & \sum_{k\neq l}^{\lfloor \log^{2}n \rfloor}|k-l|^{3|k-l|}n^{-|k-l|/6}\\
    &\leq 2\log^{2}n\sum_{m=1}^{\lfloor \log^{2}n \rfloor}m^{3m}n^{-m/6}\\
    &\leq \frac{2\log^{4}n}{n^{1/6}}=o(1).
\end{align*}

As a result, $\sum_{k\neq l = 1}^{\lfloor \log^{2}n\rfloor}z^{(-k-1)}\bar{z}^{(-l-1)}\mathbb{E}[\Tr M^{k}\Tr\bar M^{l}]=o(1).$ Therefore using Lemma \ref{lemma:equal power_complex}, we obtain
\begin{align*}
    &\mathbb{E}[\Tr \hat{\mathcal{R}}_{z}^{\circ}(M)\Tr\hat{\mathcal{R}}_{\bar\eta}^{\circ}(\bar M)]\\
    &=\left\{\mathbb{E}[\Tr \hat{\mathcal{R}}_{z}(M)\Tr\hat{\mathcal{R}}_{\bar\eta}(\bar M)]-\frac{n^{2}}{z\bar \eta}\right\}+o(1)\\
    &=\sum_{k=1}^{\lfloor \log^{2}n \rfloor}(z\bar \eta)^{-k-1}\mathbb{E}[M^{k}\bar M^{k}] + o(1)\\
    &=2\sum_{k=1}^{\lfloor \log^{2}n \rfloor}k(z\bar \eta)^{-k-1}+o(1)\\
    &=2\left(\frac{1}{z\bar \eta}\right)^{2}\left(1-\frac{1}{z\bar \eta}\right)^{-2}\\
    &=2\left(1-z\bar\eta\right)^{-2}.
\end{align*}

This completes the proof of the Proposition \ref{prop: Gaussian process and tightness}.

%%%%%%%%%%%%%%%%%%%%%%%%%%%%%%%%%%%%%%%%%%%%%%%%%%%%%%%%%%%%%%%%%%%%%%%%%%%
%%%%%%%%%%%%%%%%%%%%%%%%%%%%%%   Appendix  %%%%%%%%%%%%%%%%%%%%%%%%%%%%%%%% 
%%%%%%%%%%%%%%%%%%%%%%%%%%%%%%%%%%%%%%%%%%%%%%%%%%%%%%%%%%%%%%%%%%%%%%%%%%%

\appendix
\renewcommand\thesection{A}
\section{}\label{sec:appendix_section}
\renewcommand{\theappendixlemma}{\Alph{section}.\arabic{appendixlemma}}

\begin{appendixlemma}{\cite[Lemma A.3]{jana2022clt}}
    Let $M$ be a random band matrix with bandwidth $b_n$ and variance profile $\omega$. Assume that the entries are i.i.d. having mean zero and variance one, and satisfy the Poincar\'e inequality with constant $\alpha.$ Then there exists $\rho\geq1$ such that 
    $$
    \mathbb{P}(\|M\|>\rho/4+t) \leq K \exp\left(-\sqrt{\frac{\alpha (2b_{n}+1)}{2\omega}}t\right),\;\;\;\; \forall \;\; t>0,
    $$
    where $K$ is a constant, which does not depend on $n.$
\end{appendixlemma}
In our case, we have a matrix with i.i.d. entries whose entries satisfy the Poincar\'e inequality. Therefore, using the above lemma we can have the following.
\begin{appendixlemma}\label{lemma:Norm bound}
    Let $T$ be a  random matrix with $m \times m$ dimension having i.i.d. entries with mean zero and variance one. Moreover, assume that the entries satisfy the Poincar\'e inequality with constant $\alpha$. Then there exists $\rho\geq1$ such that 
    $$
    \mathbb{P}(\|T/\sqrt{m}\|>\rho/4+t) \leq K \exp\left(-\sqrt{\frac{\alpha m}{2}}t\right),\;\;\;\; \forall \;\; t>0,
    $$
    where $K>0$ is a universal constant.
\end{appendixlemma}

\begin{appendixlemma}{\cite[Lemma 2.2 ]{rider2006gaussian}}\label{lemma: rider silverstein bound on trace of resolvent}
    Let $M=[m_{ij}]$ be an $n\times n$ random matrix with i.i.d. complex valued random variables satisfying (i) $\mathbb{E}[m_{ij}]=0=\mathbb{E}[m_{ij}^{2}]$, (ii) $\mathbb{E}[|m_{ij}|^{k}]\leq k^{\alpha k}$ for some $\alpha > 0$ and for all $k>2$, (iii) $\Re(m_{ij})$ and $\Im(m_{ij})$ have a bounded joint density. Then for any $p\in [1, 2)$, and any $z\in\mathbb{C}$,
    \begin{align*}
        \mathbb{E}[|\Tr R_{z}(M)|^{p}]\leq C(p)n^{3p/2 + 2},
    \end{align*}
    where the constant $C(p)$ depends on $p$, but $C(p)\to\infty$ as $p\uparrow 2.$
\end{appendixlemma}

We note down the essential ingredient of this article, which martingale CLT as follows.
\begin{appendixlemma}\label{lemma:MCLT}\cite[Theorem 35.12]{billingsley2017probability}
    Let $\left\{\psi_{n, k}\right\}_{1 \leq k \leq n}$ be a martingale difference sequence with respect to a filtration $\{\mathcal{F}_{k,n}\}_{1\leq k\leq n}$. Suppose for any $\gamma>0$,\\
    \textbf{(i)} \label{MCLT i} $\lim _{n \rightarrow \infty} \sum_{k=1}^n \mathbb{E}[\psi_{n, k}^2, \mathbf{1}_{|\psi_{n, k}|>\gamma}]=0$.\\
    \textbf{(ii)} \label{MCLT ii} $\sum_{k=1}^n \mathbb{E}[\psi_{n, k}^2\left|\mathcal{F}_{n,k-1}\right] \stackrel{p}{\rightarrow} \sigma^2$ as $n \rightarrow \infty$.

    Then $\sum_{k=1}^n \psi_{n, k} \stackrel{d}\rightarrow \mathcal{N}_1\left(0, \sigma^2\right)$.
\end{appendixlemma}
\begin{appendixlemma}[Sherman-Morrison formula, \cite{sherman1950adjustment}]\label{lem: sherman morrison}
   Let $M$ and $M+ve_{k}^{t}$ be two invertible matrices, where $v\in \mathbb{C}^{n}$. Then
    \begin{align*}
    (M+v e_{k}^{t})^{-1}v=\frac{M^{-1}v}{1+e_{k}^{t}M^{-1}v}.
    \end{align*}
    \end{appendixlemma}

\subsection*{Acknowledgment} Indrajit Jana's research is partially supported by INSPIRE Fellowship\\DST/INSPIRE/04/2019/000015, Dept. of Science and Technology, Govt. of India.\\

Sunita Rani's research is fully supported by the University Grant Commission (UGC), New Delhi.

\end{document}